\documentclass[letterpaper,11pt]{article}

\newcommand{\tsum}{\textstyle{\sum}}

\usepackage{graphicx, color}   
\usepackage{multirow}
\usepackage{amsmath}
\usepackage{amsfonts}
\usepackage{latexsym}
\usepackage{amssymb, algorithm, algorithmic}
\usepackage{epstopdf}
\usepackage{epsfig}

\newcommand{\bbr}{\mathbb{R}}

\newtheorem{theorem}{Theorem}
\newtheorem{lemma}[theorem]{Lemma}
\newtheorem{corollary}[theorem]{Corollary}
\newtheorem{proposition}[theorem]{Proposition}
\newtheorem{remark}[theorem]{Remark}

\newcommand{\beq}{\begin{equation}}
\newcommand{\eeq}{\end{equation}}
\newcommand{\beqa}{\begin{eqnarray}}
\newcommand{\eeqa}{\end{eqnarray}}
\newcommand{\beqas}{\begin{eqnarray*}}
\newcommand{\eeqas}{\end{eqnarray*}}
\newcommand{\bi}{\begin{itemize}}
\newcommand{\ei}{\end{itemize}}
\newcommand{\ba}{\begin{array}}
\newcommand{\ea}{\end{array}}

\def\eqnok#1{(\ref{#1})}

\def\argmin{{\rm argmin}}
\def\argmax{{\rm argmax}}

\def\w{\omega}

\newcommand{\bbe}{\mathbb{E}}

\def\vgap{\vspace*{.1in}}

\def\E{{\bf E}}

\title{Randomized First-Order Methods for Saddle Point Optimization
\thanks{
This research was partially supported by NSF
grants CMMI-1000347, CMMI-1254446, DMS-1319050, and ONR grant N00014-13-1-0036.
}}

\author{
 Cong D. Dang
    \thanks{Department of Industrial and Systems
    Engineering, University of Florida, Gainesville, FL, 32611.
    (email: {\tt congdd@ufl.edu}). }
    \and
     Guanghui Lan
    \thanks{Department of Industrial and Systems
    Engineering, University of Florida, Gainesville, FL, 32611.
    (email: {\tt glan@ise.ufl.edu}).}
}

\RequirePackage[normalem]{ulem} 
\RequirePackage{color}\definecolor{RED}{rgb}{1,0,0}\definecolor{BLUE}{rgb}{0,0,1} 
\providecommand{\DIFdel}[1]{{\protect\color{red}\sout{}}}                      
\begin{document}

\maketitle
\begin{abstract}
In this paper, we present novel randomized algorithms for solving saddle point problems whose
dual feasible region is given by the direct product of many convex sets. Our algorithms can achieve
an ${\cal O}(1/N)$ and ${\cal O}(1/N^2)$ rate of convergence, respectively, for general bilinear saddle point
and smooth bilinear saddle point problems  based on a new prima-dual termination criterion, and each iteration
of these algorithms needs to solve only one randomly selected dual subproblem.
Moreover, these algorithms do not require strongly convex assumptions on the objective function and/or the incorporation of a strongly convex perturbation term. They do not necessarily require the primal or dual feasible regions to be bounded
or the estimation of the distance from the initial point to the set of optimal solutions to be available either.
We show that
when applied to linearly constrained problems, RPDs are equivalent to
certain randomized variants of the alternating direction method of multipliers (ADMM), while
a direct extension of ADMM does not necessarily converge when the number of blocks exceeds two.

\vspace{.1in}

\noindent {\bf Keywords.} Stochastic Optimization, Block Coordinate Descent, Nonsmooth Optimization, Saddle Point Optimization,
Alternating Direction Method of Multipliers

\vspace{.1in}

\end{abstract}

\setcounter{equation}{0}
\section{Introduction} \label{sec_intro}

Motivated by some recent applications
in data analysis, there has been a growing interest in the design and analysis of randomized first-order methods
for large-scale convex optimization. In these applications, the complex datasets
are so big and often distributed over different storage locations. It is often
impractical to assume that optimization algorithms can traverse an entire dataset once in each iteration,
because doing so is either time consuming or unreliable,
and often results in low resource utilization due to necessary synchronization among
different computing units (e.g., CPUs, GPUs, and Cores) in a distributed computing environment.
On the other hand,  randomized algorithms can make progress by using information obtained from a
randomly selected subset of data and thus provide much flexibility
for their implementation in the aforementioned distributed environments.

In this paper, we focus on the development of randomized algorithms
for solving a class of saddle point problems given by
\beq\label{spp}
\min_{x \in X} \left\{ h(x)+ \max_{y \in Y} \langle  A x,  y\rangle- J(y) \right\},
\eeq
where $X \subseteq \bbr^n$ and $Y \subseteq \bbr^m$ are closed convex sets, $h: X \to \bbr$ and $J: Y \to \bbr$
are closed convex functions, and $A: \bbr^n \to \bbr^m$ denotes a given linear operator.
Throughout this paper, we assume that
\beq \label{spp1}
y = (y_1; \ldots; y_p), \ \ Y = Y_1 \times \ldots \times Y_p, \ \ \mbox{and} \ \ J(y)=J_1(y_1) + \ldots + J_p(y_p).
\eeq
Here $y_i \in Y_i$, $i =1, \ldots,p$, $Y_i \subseteq \bbr^{m_i}$ are given closed convex sets such that $\sum_{i=1}^p m_i = m$,
and $J_i: Y_i \to \bbr$, $i=1, \ldots,p$, are closed convex functions. Accordingly, we denote
$A = (A_1; \ldots; A_p)$, where $A_i$ are given linear operators from $\bbr^n$ to $\bbr^{m_i}$, $i = 1, \ldots, p$.

Problem \eqnok{spp}-\eqnok{spp1} covers a few interesting subclasses of problems in the literature.
One prominent example is to minimize the summation of several separable convex functions over some coupled
linear constraints. Indeed, letting $X = \bbr^n$ and $h(x) = - b^T x$,
one can view problem~\eqnok{spp}-\eqnok{spp1} as the saddle-point reformation of
\beq \label{LCP_review}
\begin{array}{ll}
\min & J_1(y_1)+J_2(y_2)+\ldots+J_p(y_p) \\
s.t.& A_1^T y_1+A_2^T y_2+\ldots+ A_p^T y_p=b,\\
& y_i \in Y_i, i=1,\ldots,p.
\end{array}
\eeq

The above problem has found wide applications in machine learning and image processing,
and many first-order algorithms have developed for its solutions. More specifically,
one can
apply Nesterov's smoothing scheme~\cite{Nest05-1},
the primal-dual method~\cite{ChamPoc11-1,CheLanOu13-1}, and the mirror-prox method~\cite{Nem05-1,MonSva10-1,CheLanOu14-1}
to solve the saddle-point reformulation in \eqnok{spp}.
We can also apply some classic penalty-based approaches for solving \eqnok{LCP_review}.
In particular, Lan and Monteiro
discussed the complexity of first-order quadratic penalty methods~\cite{LanMon13-1} and
augmented Lagrangian penalty methods~\cite{LanMon09-1} applied to problem~\eqnok{LCP_review}.
More recently, He, Juditsky and Nemirovski generalized the mirror-prox algorithm for solving problem~\eqnok{LCP_review}
based on the exact penalty method~\cite{HeJudNem14-1}. When $p = 2$,
a special augmented Lagrangian method, namely the alternating direction method of
multipliers (ADMM)~\cite{DouglasRachford56-1,Rocka76-1,LiMe79-1,Gabay83-1,EcBe92-1}, has been intensively studied
recently~\cite{boyd2011distributed,GoMa12-1,GoMaSc13-1,HeYu12-1,OuCheLanPas14-1}.
However, as shown by Chen et al. \cite{ChenHeYeYuan13-1}, a direction extension of ADMM
does not necessarily converge when $p > 2$, unless some strong convexity assumptions on
$J_i$ and full row rank assumption on $A_i$ are made (e.g., \cite{HeYu12-1,HongLuo13-1,LinMaZhang14-1}).
Observe that all these methods need to perform $p$ projection subproblems over the sets $Y_i$, $i = 1, \ldots, p$,
in every iteration.

Another interesting example is to minimize the regularized loss function given by
\beq \label{reg_loss}
\min_{x \in X} h(x) + \sum_{i=1}^p f_i(A_i x),
\eeq
where $f_i: \bbr^{m_i} \to \bbr$ are closed convex functions
with conjugate $f_i^*$, $i = 1, \ldots, p$.
Clearly, problem \eqnok{reg_loss} can be viewed as a special case of problem \eqnok{spp}
with $J_i = f^*_i$ and $Y_i = \bbr^{m_i}$, $i=1, \ldots, p$.
While the algorithms for solving problem \eqnok{LCP_review} are mostly deterministic,
much effort has been devoted to randomized first-order methods
for solving problem \eqnok{reg_loss}, which can make progress by utilizing
the (sub)gradient of a randomly selected component $f_i(A_i x)$ only.
More specifically, if $f_i$ are general nonsmooth convex functions,
one can apply the mirror-descent stochastic approximation in
\cite{NJLS09-1} or the accelerated stochastic approximation in \cite{Lan10-3},
which exhibit an ${\cal O}(1/\sqrt{N})$ rate of convergence for solving problem \eqnok{reg_loss}. Here
$N$ denotes the number of iterations.
Recently, some interesting development has been made \cite{SchRouBac13-1,BlHeGa07-1, MahdaviJ13,Suzuki13}
under the assumption that $f_i$ are smooth convex functions.
Based on incremental averaging gradient method~\cite{BlHeGa07-1},
Schmidt et. al.~\cite{SchRouBac13-1} developed a stochastic averaging gradient method
and show that it exhibits an ${\cal O}(1/N)$ rate of convergence for smooth problems and
an linear rate of convergence for the case when $f_i$ are smooth and strongly convex.
This algorithm is also closely related to the stochastic dual coordinate ascent \cite{ShalevZhang13-1}, a randomized
version of dual coordinate ascent applied to the dual of problem \eqnok{reg_loss}
when $h$ is strongly convex, see~\cite{Nest10-1,LevLew10-1,Nest12-1,Rich12-1,BeckTet13-1, LuXiao13-1, DangLan13-1}
for some recent developments on block coordinate descent methods.

In this paper, we propose a novel algorithm, namely the randomized primal-dual method, to solve problems
in the form of \eqnok{spp}-\eqnok{spp1}. The main idea is to incorporate a block decomposition of dual space into the
primal-dual algorithm in \cite{ChamPoc11-1}. At each iteration, our algorithm requires to solve only one subproblem
in dual space instead of $p$ subproblems as in the primal-dual algorithm. By using a
new primal-dual termination
criterion inspired by the one employed by Monteiro and Svaiter \cite{MonSva10-3}, we show that our algorithm can achieve
an ${\cal O}(1/N)$ and ${\cal O}(1/N^2)$ rate of convergence, respectively
for solving general bilinear saddle point problems (without any strongly convex assumptions) and smooth bilinear saddle point problems
(with $J$ being strongly convex), where $N$ is the number of iterations.
Furthermore, we demonstrate that our algorithm can deal with the situation
when either $X$ or $Y$ is unbounded, as long as a saddle point of problem \eqnok{spp}-\eqnok{spp1} exists.
It should be noted that these complexity results will have an extra constant factor which depends on
the number of blocks $p$, but such a dependence is mild if $p$ is not too big.
In addition, we discuss possible
extensions of the RPD method to the non-Euclidean geometry and also show that
RPD applied to the linearly constrained problems in \eqnok{LCP_review} is
equivalent to a certain randomized variant of the ADMM method. To the best of our knowledge,
all these developments seem to be new in the literature.
In fact, our proof for the convergence of the ergodic mean of
the primal-dual method for smooth bilinear saddle point problems was also new even under the deterministic setting (i.e., $p = 1$),
\footnote{It is worth noting that Chambolle and Pock~\cite{ChamPoc14-1} had also released their results on the convergence
of the ergodic means for deterministic primal-dual methods shortly after we released the initial version of the current paper in Sep., 2014.}.

It should be noted that in a concurrent and independent work,
Zhang and Xiao~\cite{Yuchen14} presented a randomized version of the primal-dual method
for solving a special class of regularized
empirical risk minimization (ERM) problems given in the form of \eqnok{reg_loss}
\footnote{Note that \cite{Yuchen14}
was also initially released in Sep., 2014.}.
However, the algorithms, analysis and termination criteria in these papers are significantly different:
(a) our primal-dual algorithm does not involve any extrapolation step
as used in \cite{Yuchen14};  (b) we employed a new primal-dual optimality gap to assess the quality of
a feasible solution to problem~\eqnok{spp}, while \cite{Yuchen14} employs the distance
to the optimal solution as the termination criterion; and (c) as a consequence, the convergence analyses
in these papers are significantly different. In fact, the basic algorithm in \cite{Yuchen14}
was designed for problems where $h$ is strongly convex problems (similarly to those randomized dual coordinate
descent methods~\cite{ShaZhang15-1}). Otherwise, one has to add
a strongly convex perturbation to the objective function and impose stronger assumptions about
$f_i$ and $h$. Such a perturbation term
can be properly chosen only if there exists a bound on the distance from the initial
point to the set of optimal solutions, and hence are not best suitable for the linearly constrained
problems in \eqnok{LCP_review}. In fact, the authors were not aware of the existence of
any other randomized algorithms in the literature that do not require the incorporation of
a perturbation term for solving \eqnok{spp}-\eqnok{spp1}, but can achieve
the optimal rate of convergence in terms of their dependence on $N$ as shown in this paper.

This paper is organized as follows. We first discuss some new primal-dual termination criteria in
Section 2. We then present a general RPD method in Section 3, and
discuss its convergence properties for general bilinear saddle point and smooth bilinear saddle
point problems under the assumption that the primal and dual feasible regions are bounded.
In Section 3, we generalize the RPD method for the case when the feasible regions are unbounded
and incorporate non-Euclidean distance generating functions into the RPD method.
In Section 4, we discuss the relation of the RPD method to ADMM. Finally some brief concluding
remarks are provided in Section 5.

\section{The problem of interest and its termination criteria}
We introduce in this section a few termination criteria that will be used to evaluate the solution quality
for problem \eqnok{spp}.

Denote $Z \equiv X \times Y$. For a given $\hat z=(\hat x, \hat y) \in Z$,
let us define the gap function $Q_0$ by
\beq \label{gap_complete}
Q_0(\hat z, z):=\left[h(\hat x)+\left\langle {A\hat x , y}\right\rangle-J(y)\right]-\left[h(x)+\left\langle {Ax, \hat y}\right\rangle-J(\hat y) \right], \ \ \
\forall  z=(x,y) \in Z.
\eeq
It can be easily verified that $\hat z \in Z$ is an optimal solution of problem \eqnok{spp}-\eqnok{spp1}
if and only if $Q_0(\hat z, z) \le 0$ for any $z \in Z$. A natural way to assess the solution
quality of $\hat z$ is to compute the gap
\beq \label{def_g0}
g_0(\hat z) = \max_{z \in X} Q_0(\hat z, z),
\eeq
under the assumption that $g_0$ is well-defined, e.g., when $Z$ is bounded \cite{ChamPoc11-1,CheLanOu13-1}.
Since $\hat z$ is a random variable in the randomized primal-dual algorithm to be studied in this paper,
one would expect to use $\E[g_0(\hat z)]$ to measure the quality of $\hat z$. However, except for a few specific cases,
we cannot provide an error bound on $\E[g_0(\hat z)]$ in general. Instead, we will introduce a
slightly relaxed termination criterion defined as follows. For any given $\delta \in \bbr$, let us denote
\beq \label{gap_complete1}
Q_\delta(\hat z, z):=\left[h(\hat x)+\left\langle {A\hat x , y}\right\rangle-J(y)\right]-\left[h(x)+\left\langle {Ax, \hat y}\right\rangle-J(\hat y) \right] + \delta, \ \ \
\forall  z=(x,y) \in Z
\eeq
and
\beq \label{def_gdelta}
g_\delta(\hat z) := \max_{z \in X} Q_\delta(\hat z, z).
\eeq
We will show the convergence of
the randomized primal-dual algorithm in terms of the expected primal-dual gap $\E [g_\delta(\hat z)]$
for some $\delta \in \bbr$ satisfying $\E[\delta] = 0$. Clearly, $g_0$ in \eqnok{def_g0}
is a specialized version of $g_\delta$ with $\delta = 0$.

One potential problem associated with the aforementioned primal-dual gap
$g_\delta$ is that it is not well-defined if $Z$ is unbounded.
In the latter case, Monteiro and Svaiter~\cite{MonSva09-1}
suggested a perturbation-based termination criterion for solving problem \eqnok{spp}-\eqnok{spp1}
inspired by the enlargement of a maximal monotone operator that was first
studied in \cite{burachik1997enlargement}. One advantage of using this criterion is that its definition does not
depend on the boundedness of the domain of the operator. More specifically,
as shown in \cite{MonSva09-1}, there always exists a perturbation vector $v$ such that
\[
\tilde g_0(\hat z, v) := \max_{z \in Z} Q_0(\hat z, z) - \langle v, \hat z - z \rangle
\]
is well-defined, although the value of $g_0(\hat z)$ in \eqnok{def_g0} may be unbounded if $Z$ is unbounded.
Accordingly, for the case when $\hat z$ is a random variable, we define
\beq \label{def_tgdelta}
\tilde g_\delta(\hat z, v) := \max_{z \in Z} Q_\delta(\hat z, z) - \langle v, \hat z - z \rangle
\eeq
and establish the convergence of the randomized primal-dual algorithm in terms of $\E[\tilde g_\delta(\hat z, v)]$
for some $\delta \in \bbr$ satisfying $\E[\delta] = 0$.

\section{The algorithm and main results}
This section consists of three subsections. We first present a generic
randomized primal-dual (RPD) method in subsection 3.1, and discuss its convergence
properties for solving different classes of saddle point problems in the two subsequent
subsections. More specifically, we focus on the analysis of the RPD method
for solving general saddle point problems, where both $h$ and $J$ are general convex functions
without assuming strong convexity, over bounded feasible sets in subsection 3.2.
We then show in subsection 3.3 that much stronger convergence
properties can be obtained for solving smooth saddle point problems, for which $J$ is strongly convex.
It is worth noting that the same algorithmic framework presented in subsection 3.1
is applicable to all these different cases mentioned above, as well as the unbounded
case to be discussed in Section 4.

\subsection{The RPD algorithm}
We will first introduce a few notations in order to simplify the
description and analysis of the RPD algorithm.
Let $I_m$ and $I_{m_i}, i=1,2,\ldots,p$, respectively, denote the identity matrices
in $\bbr^{m\times m}$ and $\bbr^{m_i\times m_i}, i=1,2,\ldots,p.$ Observe that
$I_{m_i}, i=1,2,\ldots,p,$ can be viewed as the $i$-th diagonal block of $I_m$.
Also let us define $U_i \in \bbr^{m\times m}$, $i=1,2,\ldots,p$, as the diagonal
matrix whose $i$-th diagonal block is $I_{m_i}$ and all other blocks are given by $0$.
Also let $\bar U_i \in \bbr^{m\times m}$ be the complement of $U_i$ such that
\[
U_i + \bar U_i = I_m.
\]

With the help of the above notations, we are now ready to describe our algorithmic framework as follows.
\begin{algorithm} [H]
    \caption{The randomized primal-dual (RPD) method for saddle point optimization}
    \label{algRPD}
    \begin{algorithmic}
\STATE Let $z^1=(x^1,y^1) \in X\times Y$,
and nonnegative stepsizes $\{\tau_t\},$ $\{\eta_t\},$ parameters $\{q_t\}$, and weights $\{\gamma_t\}$
be given. Set $\bar x^1 = x^1$.

\FOR {$t=1, \ldots,N$}

\STATE 1. Generate a random variable $i_t$ uniformly distributed over $\{1,2,...,p \}.$

\STATE 2. Update $y^{t+1}$ and $x^{t+1}$ by
\begin{align}
y_i^{t+1} &=
\begin{cases}
\argmin_{y_{i} \in Y_{i}}{\left\langle {-U_{i}A\bar x^t,y}\right\rangle+J_{i} (y_{i})+\tfrac{\tau_t}{2} \| y_{i}-y_{i}^t\|_2^2}, & i = i_t,\\
y_i^{t}, & i \ne i_t.
\end{cases}  \label{eqn_update_y}\\
x^{t+1}&=\argmin_{x \in X} {h(x)+\left\langle {x,A^Ty^{t+1}}\right\rangle}+\tfrac{\eta_t}{2} \| x-x^t\|_2^2. \label{eqn_update_x}\\
\bar x^{t+1}&=q_t (x^{t+1}-x^t)+ x^{t+1}. \label{eqn_x_bar}
\end{align}
\ENDFOR

{\bf Output:}  Set \beq \label{defxbar}
\hat z^N=\left(\tsum_{t=1}^{N-1}{\gamma_t}\right)^{-1}\tsum_{t=1}^{N-1}{\gamma_t} z^{t+1}.
\eeq

    \end{algorithmic}
\end{algorithm}

The above RPD algorithm originated from the primal-dual method in \cite{ChamPoc11-1}. The major differences between these two algorithms
are summarized as follows. Firstly, instead of updating the whole dual variable $y_i^t$, $i = 1, \ldots, p$, as in the original primal-dual algorithm,
the RPD algorithm updates in Step \eqnok{eqn_update_y} the $i_t$-th component of $y^t$ only. Secondly, rather than using constant stepsizes for $\tau_t$, $\eta_t$, and $q_t$,
variable stepsizes are used in the RPD method. Thirdly, the output solution $\hat z^N$ is defined as a weighted average rather than a
simple average of $z^t$, $t=2, \ldots, N+1$. The latter two enhancements are introduced so that the primal-dual algorithm can achieve the optimal
rate of convergence for solving smooth saddle point problems, which is new even for the deterministic case where the number of blocks $p = 1$.

It is also known that the primal-dual algorithm is related to
the Douglas-Rachford splitting method \cite{DouglasRachford56-1,LiMe79-1} and a pre-conditioned version of
the alternating direction method of multipliers \cite{gabay1976dual,glowinski1975approximation}
(see, e.g., \cite{boyd2011distributed,ChamPoc11-1,esser2010general,he2012on,monteiro2013iteration} for detailed reviews on the relationship between
the primal-dual methods and other algorithms, as well as recent theoretical developments). However, to the best of our knowledge,
there does not exist randomized version of these algorithms which only need to solve one dual subproblem at each iteration
before in the literature (see Section 4 for more discussions).

It should be noted that Algorithm~\ref{algRPD} is conceptual only since we have not yet specified a few algorithmic parameters
including  $\{\tau_t\}$, $\{\eta_t\}$, $\{q_t\}$, and $\{\gamma_t\}$.
We will come back to this issue after establishing some convergence properties of the generic RPD method
for solving different classes of saddle-point problems.

\subsection{General bilinear saddle point problems over bounded feasible sets}

Throughout this subsection we assume that both $h$ and $J$ are general convex function (without assuming strong convexity)
so that problems \eqnok{eqn_update_y} and \eqnok{eqn_update_x} are relatively easy to solve.
Also we assume that both $X$ and $Y$ are bounded, i.e., $\exists$ $\Omega_X > 0$ and $\Omega_Y >0$ such that
\beq\label{omega}
\max_{x_1,x_2 \in X} \|x_1-x_2 \|_2^2 \le \Omega_X^2 \; \mbox{and} \; \max_{y_1,y_2 \in Y} \|y_1-y_2 \|_2^2 \le \Omega_Y^2.
\eeq

Before establishing the main convergence properties for the RPD method applied to general bilinear saddle point problems,
we show an important recursion of this algorithm in the following result.

\begin{proposition} \label{proposition1}
Let $z^t=(x^t,y^t), t=1,2,\ldots,N,$ be generated by Algorithm~\ref{algRPD}. For any $z \in Z,$ we have
\beq \label{recursion1}
\begin{array}{ll}
&\gamma_t Q_0(z^{t+1},z)+\left\langle {\gamma_tAx^{t+1}-Ax^t, y^{t+1}-y}\right\rangle+ (\gamma_t-1)\left[ J(y)-J(y^{t+1})\right]-\Delta _t \\
&\le \tfrac{\gamma_t\eta_t}{2}\left[\| x-x^t\|_2^2- \| x^t-x^{t+1}\|_2^2-\| x-x^{t+1}\|_2^2\right]\\
&\;\;\;\;+\tfrac{\tau_t}{2}\left[\| y-y^t\|_2^2-\| y-y^{t+1}\|_2^2-\|y^t-y^{t+1} \|_2^2\right],
\end{array}
\eeq
where
\beq \label{def_triangle_t}
\begin{array}{ll}
\Delta _t:=\left\langle {q_{t-1}U_{i_t} A(x^t-x^{t-1}), y^{t+1}-y}\right\rangle-\left\langle {\bar U_{i_t}Ax^t, y^{t}-y}\right\rangle+\tsum_{i \ne i_t}[J (y_{i}^{t})-J_i(y_i)].
\end{array}
\eeq
\end{proposition}
\begin{proof}
By the optimality condition of problem~\eqnok{eqn_update_x}, for all $x \in X,$ we have
\beq \label{observation2}
h(x^{t+1})-h(x)+\left\langle {x^{t+1}-x, A^Ty^{t+1}}\right\rangle
+\tfrac{\eta_t}{2} \| x^t-x^{t+1}\|_2^2+\tfrac{\eta_t}{2} \| x-x^{t+1}\|_2^2 \le \tfrac{\eta_t}{2} \| x-x^t\|_2^2.
\eeq
Observe that
$$
\begin{array}{ll}
\left\langle {x^{t+1}-x, A^Ty^{t+1}}\right\rangle&=\left\langle {Ax^{t+1}, y}\right\rangle-\left\langle {Ax, y^{t+1}}\right\rangle-\left\langle {Ax^{t+1}, y}\right\rangle+\left\langle {Ax^{t+1}, y^{t+1}}\right\rangle\\
 &= \left\langle {Ax^{t+1}, y}\right\rangle-\left\langle {Ax, y^{t+1}}\right\rangle+\left\langle {Ax^{t+1}, y^{t+1}-y}\right\rangle,
\end{array}
$$
which together with \eqnok{observation2} and the definition of $Q_0$ in \eqnok{gap_complete} then imply
\beq \label{observation3}
\begin{array}{ll}
&Q_0(z^{t+1},z)+\left\langle {Ax^{t+1}, y^{t+1}-y}\right\rangle+J(y)-J(y^{t+1})\\
\le& \tfrac{\eta_t}{2}\left[\| x-x^t\|_2^2- \| x^t-x^{t+1}\|_2^2-\| x-x^{t+1}\|_2^2\right].
\end{array}
\eeq
Now, by the optimality condition of problem~\eqnok{eqn_update_y}, for all $y \in Y,$ we have
\beq \label{observation1}
\left\langle {-U_{i_t}A{\bar x}^t, y^{t+1}-y}\right\rangle +J_{i_t}(y_{i_t}^{t+1})-J_{i_t}(y_{i_t})
+\tfrac{\tau_t}{2} \| y_{i_t}^t-y_{i_t}^{t+1}\|_2^2+\tfrac{\tau_t}{2} \| y_{i_t}-y_{i_t}^{t+1}\|_2^2 \le \tfrac{\tau_t}{2} \| y_{i_t}-y_{i_t}^t\|_2^2.
\eeq
Using the definition of $\bar x^t$ in  \eqnok{eqn_x_bar}, we also have
\beq \label{observation111}
\begin{array}{ll}
\left\langle {-U_{i_t}A{\bar x^t}, y^{t+1}-y}\right\rangle &=\left\langle {-U_{i_t}A[q_{t-1} (x^t-x^{t-1})+x^t], y^{t+1}-y}\right\rangle\\
&=\left\langle {-U_{i_t}Ax^t, y^{t+1}-y}\right\rangle-\left\langle {q_{t-1}U_{i_t} A(x^t-x^{t-1}), y^{t+1}-y}\right\rangle\\
&=\left\langle {-(U_{i_t}+\bar U_{i_t})Ax^t, y^{t+1}-y}\right\rangle-\left\langle {q_{t-1}U_{i_t} A(x^t-x^{t-1}), y^{t+1}-y}\right\rangle\\
&\quad+\left\langle {\bar U_{i_t}Ax^t, y^{t+1}-y}\right\rangle\\
&=\left\langle {-Ax^t, y^{t+1}-y}\right\rangle-\left\langle {q_{t-1}U_{i_t} A(x^t-x^{t-1}), y^{t+1}-y}\right\rangle+\left\langle {\bar U_{i_t}Ax^t, y^{t}-y}\right\rangle,
\end{array}
\eeq
where the last identity follows from the fact that $U_{i_t}+\bar U_{i_t}=I_n$ and
that $\left\langle {\bar U_{i_t}Ax^{t}, y^{t+1}-y}\right\rangle=\left\langle {\bar U_{i_t}Ax^t, y^{t}-y}\right\rangle.$
Also observe that
\beq \label{observation112}
\begin{array}{ll}
J_{i_t} (y_{i_t}^{t+1})-J_{i_t} (y_{i_t})&=J (y^{t+1})-\\sum_{i \ne i_t}J_{i} (y_{i}^{t})-[J (y)- \tsum_{i \ne i_t}J_{i} (y_{i})]\\
&=J (y^{t+1})-J(y)-\tsum_{i \ne i_t}[J (y_{i}^{t})-J_i(y_i)],\\
\| y_{i_t}^t-y_{i_t}^{t+1}\|_2^2&=\|y^t-y^{t+1} \|_2^2,\\
\| y_{i_t}-y_{i_t}^t\|_2^2-\| y_{i_t}-y_{i_t}^{t+1}\|_2^2&=\| y-y^t\|_2^2-\| y-y^{t+1}\|_2^2.
\end{array}
\eeq
Using these observations in \eqnok{observation1}, we conclude
$$\begin{array}{ll}
&\left\langle {-Ax^t, y^{t+1}-y}\right\rangle-\left\langle {q_{t-1}U_{i_t} A(x^t-x^{t-1}), y^{t+1}-y}\right\rangle+\left\langle {\bar U_{i_t}Ax^t, y^{t}-y}\right\rangle\\
&+J (y^{t+1})-J(y)-\tsum_{i \ne i_t}[J (y_{i}^{t})-J_i(y_i)] \le \frac{\tau_t}{2}\left[\| y-y^t\|_2^2-\| y-y^{t+1}\|_2^2-\|y^t-y^{t+1} \|_2^2\right].
\end{array}
$$
Multiplying both sides of \eqnok{observation3} by $\gamma_t$ and adding it up with the above inequality, we have
$$\begin{array}{ll}
&\gamma_tQ_0(z^{t+1},z)+\left\langle {\gamma_tAx^{t+1}-Ax^t, y^{t+1}-y}\right\rangle+ (\gamma_t-1)\left[ J(y)-J(y^{t+1})\right]\\
&-\left\langle {q_{t-1}U_{i_t} A(x^t-x^{t-1}), y^{t+1}-y}\right\rangle+\left\langle {\bar U_{i_t}Ax^t, y^{t}-y}\right\rangle-\tsum_{i \ne i_t}[J (y_{i}^{t})-J_i(y_i)] \\
&\le \frac{\gamma_t\eta_t}{2}\left[\| x-x^t\|_2^2- \| x^t-x^{t+1}\|_2^2-\| x-x^{t+1}\|_2^2\right]\\
&\;\;\;\;+\frac{\tau_t}{2}\left[\| y-y^t\|_2^2-\| y-y^{t+1}\|_2^2-\|y^t-y^{t+1} \|_2^2\right],
\end{array}
$$
which, in view of the definition of $\Delta _t$, clearly implies the result.
\end{proof}

\vgap

The following lemma provides an upper bound on $\E_{i_t}[\Delta_t].$
\begin{lemma} \label{lemma1}
Let $\Delta_t$ be defined in \eqnok{def_triangle_t}.
If $i_t$ is uniformly distributed on $\{1,2,...,p\},$ then
$$\begin{array}{ll}
\E_{i_t}[\Delta _t] \le& \left\langle{\left( \tfrac{1}{p}q_{t-1}-\tfrac{p-1}{p} \right) Ax^t-\tfrac{1}{p}q_{t-1}Ax^{t-1}, y^{t}-y}\right\rangle
+\tfrac{p-1}{p} \left[ J (y^{t})-J(y)\right]\\
&+\tfrac{q_{t-1}^2 \|A\|_2^2}{2p\tau_t}\| x^t-x^{t-1}\|_2^2+\tfrac{\tau_t}{2}\E_{i_t} \left[\| y^{t+1}-y^t\|_2^2\right].
\end{array}
$$
\end{lemma}

\begin{proof}
The definition of $\Delta _t$ in \eqnok{def_triangle_t} can be rewritten as
\beq \label{def_delat_a}
\begin{array}{ll}\Delta _t&=\left\langle {q_{t-1}U_{i_t} A(x^t-x^{t-1})-\bar U_{i_t}Ax^t, y^{t}-y}\right\rangle\\
&\;\;\;\;-\left\langle {q_{t-1}U_{i_t} A(x^t-x^{t-1}), y^{t+1}-y^t}\right\rangle+\tsum_{i \ne i_t}[J_i (y_{i}^{t})-J_i(y_i)].\end{array}
\eeq
Since $i_t$ is uniformly distributed on $\{1,2,...,p\},$ we have
\beq\label{recur_expectation}
\begin{array}{ll}
&\E_{i_t}\left[\left\langle {q_{t-1}U_{i_t} A(x^t-x^{t-1})-\bar U_{i_t}A}x^t, y^{t}-y\right\rangle\right]\\
&=\left\langle{\tfrac{1}{p}q_{t-1}A(x^t-x^{t-1}), y^{t}-y}\right\rangle-\tfrac{p-1}{p}\left\langle {Ax^t, y^{t}-y}\right\rangle]\\
&=\left\langle{\left( \tfrac{1}{p}q_{t-1}-\tfrac{p-1}{p} \right) Ax^t-\tfrac{1}{p}q_{t-1}Ax^{t-1}, y^{t}-y}\right\rangle
\end{array}
\eeq
and
\beq \label{norm_equal}
\E_{i_t} \left[ \tsum_{i \ne i_t}\left(J_i (y_{i}^{t})-J_i(y_i)\right)\right]=\tfrac{p-1}{p} \left[ J (y^{t})-J(y)\right].
\eeq
Observe that
\beq \label{cauchy}
\begin{array}{ll}
\E_{i_t} \left[\left\langle {q_{t-1}U_{i_t}A(x^t-x^{t-1}),y^{t+1}-y^t}\right\rangle \right]&\le \E_{i_t} \left[q_{t-1} \| U_{i_t}A( x^t-x^{t-1})\|_2 \| y^{t+1}-y^t\|_2\right]\\
&\le \E_{i_t} \left[\frac{q_{t-1}^2}{2\tau_t}\| U_{i_t}A(x^t-x^{t-1})\|_2^2+\frac{\tau_t}{2}\| y^{t+1}-y^t\|_2^2\right]\\
&=  \frac{q_{t-1}^2 }{2p\tau_t}\| A(x^t-x^{t-1})\|_2^2+\frac{\tau_t}{2}\E_{i_t} \left[\| y^{t+1}-y^t\|_2^2\right]\\
&\le \frac{q_{t-1}^2 \|A\|_2^2}{2p\tau_t}\| x^t-x^{t-1}\|_2^2+\frac{\tau_t}{2}\E_{i_t} \left[\| y^{t+1}-y^t\|_2^2\right],
\end{array}
\eeq
where the second inequality follows from the Cauchy-Swartz inequality and
the equality follows from the fact that $i_t$ is uniformly distributed on $\{1,2,...,p\}$.
The result immediately follows from \eqnok{def_delat_a}, \eqnok{recur_expectation}, \eqnok{norm_equal}, and
\eqnok{cauchy}.
\end{proof}

\vgap

We are now ready to establish the main convergence properties of the RPD algorithm
for solving saddle point problems over bounded feasible sets.

\begin{theorem} \label{theorem1}
Suppose that the initial point of Algorithm~\ref{algRPD} is chosen such that $x^1=x^0$ and $y^1=\argmax_{y \in Y} \langle A x^1, y \rangle - J(y).$
Also assume that the parameters $\{q_t\}$, $\{\gamma_t\}$, $\{\tau_t\}$, and $\{\eta_t\}$ satisfy
\begin{align}
q_t &= p, \; t=1,..,N-1, \label{stepsize_q1a}\\
\gamma_{t}&=\tfrac{1}{p} q_t-\tfrac{p-1}{p},t=1,...,N-2 \; \mbox{and} \; \gamma_{N-1}=1, \label{stepsize_q1b} \\
\tau_{t-1}&\ge \tau_{t}, \; i=1,..,N-1, \label{stepsize_q2a}\\
\gamma_{t-1}\eta_{t-1}&\ge \gamma_{t}\eta_{t}, \; i=1,..,N-1, \label{stepsize_q2b}\\
p\gamma_t\eta_t\tau_{t+1}&\ge q_{t}^2\|A\|_2^2, \; i=1,..,N-2, \label{stepsize_q3a}\\
\gamma_{N-1}\eta_{N-1}\tau_{N-1} &\ge \|A\|_2^2. \label{stepsize_q3b}
\end{align}
\begin{itemize}
\item [a)]
For any $N \ge 1,$ we have
\beq \label{convergence}
\E[Q_0(\hat z^{N},z)] \le \left( \tsum_{t=1}^{N-1}\gamma_t\right)^{-1} \left [\tfrac{\gamma_{1}\eta_{1}}{2}\Omega _X^2+\tfrac{\tau_{1}}{2}\Omega _Y^2
 \right ], \; \forall z \in Z,
\eeq
where $\hat z^{N}$ is defined in \eqnok{defxbar} and the expectation is taken w.r.t. $[i_{N}]=(i_1,...,i_{N-1})$.
\item [b)]
For any $N \ge 1,$ there exists a function $\delta(y)$ such that $\bbe[\delta(y)] = 0$ for any $y \in Y$ and
\beq \label{cor4}
\E[g_{\delta(y)}(\hat z^N)] \le
 \left( \tsum_{t=1}^{N-1}\gamma_t\right)^{-1} \left [\tfrac{\gamma_{1}\eta_{1}}{2}\Omega _X^2+\tfrac{\tau_{1}}{2}\Omega _Y^2 \right ], \; \forall z \in Z.
\eeq
\end{itemize}
\end{theorem}

\begin{proof} 
We first show part a).
It follows from Proposition~\ref{proposition1} and Lemma~\ref{lemma1} that
\beq\label{recursion1_new}
\begin{array}{ll}
&\gamma_tQ_0(z^{t+1},z)+\left\langle {\gamma_tAx^{t+1}-Ax^t, y^{t+1}-y}\right\rangle+ (\gamma_t-1)\left[ J(y)-J(y^{t+1})\right] \\
&\le \E_{i_t}[\Delta _t]+\frac{\gamma_t\eta_t}{2}\left[\| x-x^t\|_2^2- \| x^t-x^{t+1}\|_2^2-\| x-x^{t+1}\|_2^2\right]\\
&\quad +\frac{\tau_t}{2}\left[\| y-y^t\|_2^2-\| y-y^{t+1}\|_2^2-\|y^t-y^{t+1} \|_2^2\right]+\Delta _t-\E_{i_t}[\Delta _t]\\
& \le\left\langle{\left( \tfrac{1}{p}q_{t-1}-\tfrac{p-1}{p} \right) Ax^t-\tfrac{1}{p}q_{t-1}Ax^{t-1}, y^{t}-y}\right\rangle+\frac{p-1}{p} \left[ J (y^{t})-J(y)\right]\\
&\quad +\frac{q_{t-1}^2 \|A\|_2^2}{2p\tau_t}\| x^t-x^{t-1}\|_2^2+\frac{\tau_t}{2}\E_{i_t} \left[\| y^{t+1}-y^t\|_2^2\right]+\frac{\gamma_t\eta_t}{2}\left[\| x-x^t\|_2^2- \| x^t-x^{t+1}\|_2^2-\| x-x^{t+1}\|_2^2\right]\\
&\quad +\frac{\tau_t}{2}\left[\| y-y^t\|_2^2-\| y-y^{t+1}\|_2^2-\|y^t-y^{t+1} \|_2^2\right]+\Delta _t-\E_{i_t}[\Delta _t].
\end{array}
\eeq
Denoting
$$\Delta'_t :=\Delta_t -\frac{\tau_t}{2}\| y^t-y^{t+1}\|_2^2,$$
we can rewrite \eqnok{recursion1_new} as
\beq\label{recursion1_new2}
\begin{array}{ll}
&\gamma_tQ_0(z^{t+1},z)+\left\langle {\gamma_tAx^{t+1}-Ax^t, y^{t+1}-y}\right\rangle+ (1-\gamma_t)\left[ J(y^{t+1})-J(y)\right] \\
&\le\left\langle{\left( \tfrac{1}{p}q_{t-1}-\tfrac{p-1}{p} \right) Ax^t-\tfrac{1}{p}q_{t-1}Ax^{t-1}, y^{t}-y}\right\rangle
+\tfrac{p-1}{p} \left[ J (y^{t})-J(y)\right]-\tfrac{\gamma_t\eta_t}{2}\| x^t-x^{t+1}\|_2^2\\
&\quad + \, \tfrac{q_{t-1}^2 \|A\|_2^2}{2p\tau_t}\| x^t-x^{t-1}\|_2^2+\tfrac{\gamma_t\eta_t}{2}\left[\| x-x^t\|_2^2-\| x-x^{t+1}\|_2^2\right]
+\tfrac{\tau_t}{2}\left[\| y-y^t\|_2^2-\| y-y^{t+1}\|_2^2\right]\\
&\quad+ \, \Delta'_t-\E_{i_t}[\Delta'_t].
\end{array}
\eeq
Taking summation from $t=1$ to $N-1$ on both sides of the above inequality, using the assumptions in \eqnok{stepsize_q1a} and \eqnok{stepsize_q1b},
and denoting $z^{[N]}:=\{ (x^t,y^t)\}_{t=1}^N$ and
\beq \label{defB}
{\cal B}_N(z,z^{[N]}):=\tsum_{t=1}^{N-1} \left[ \tfrac{\gamma_t\eta_t}{2}\| x-x^t\|_2^2-\tfrac{\gamma_t\eta_t}{2}\| x-x^{t+1}\|_2^2\right]
+\tsum_{t=1}^{N-1} \left[ \tfrac{\tau_t}{2}\| y-y^t\|_2^2-\tfrac{\tau_t}{2}\| y-y^{t+1}\|_2^2\right],
\eeq
we then conclude that
\beq \label{summation}
\begin{array}{ll}
&\tsum_{t=1}^{N-1}\gamma_tQ_0(z^{t+1},z) \\
&\le {\cal B}_N(z,z^{[N]})-\left\langle{Ax^{N}-Ax^{N-1}, y^{N}-y}\right\rangle+\left\langle{\tfrac{1}{p} Ax^1-Ax^0, y^{1}-y}\right\rangle+\tfrac{p-1}{p} \left[ J (y^{1})-J(y)\right]\\
&\quad +\, \tfrac{p \|A\|_2^2}{2\tau_1}\| x^1-x^0\|_2^2-\tfrac{\gamma_{N-1}\eta_{N-1}}{2} \| x^N-x^{N-1}\|_2^2
- \tsum_{t=1}^{N-2}\left(\tfrac{\gamma_t\eta_t}{2} -\frac{q_{t}^2\|A\|_2^2}{2p\tau_{t+1} }\right) \| x^{t+1}-x^{t}\|_2^2\\
&\quad +\, \tsum_{t=1}^{N-1}\left(\Delta'_t-\E_{i_t}[\Delta'_t]\right)\\
&\le  {\cal B}_N(z,z^{[N]})-\tfrac{\gamma_{N-1}\eta_{N-1}}{2} \| x^N-x^{N-1}\|_2^2-\left\langle {Ax^{N}-Ax^{N-1}, y^{N}-y}\right\rangle+\tsum_{t=1}^{N-1}\left(\Delta'_t-\E_{i_t}[\Delta'_t]\right),
\end{array}
\eeq
where the second inequality follows from \eqnok{stepsize_q3a}, and the facts that $x^1=x^0$ and $y^1=\argmax_{y \in Y} \langle A x^1, y \rangle - J(y).$
Using the above conclusion, the definition of $\hat x^N$ in \eqnok{defxbar}, and the convexity of $Q_0(\hat z, z)$ w.r.t. $\hat z$, we obtain
$$
\begin{array}{ll}
\left(\tsum_{t=1}^{N-1}\gamma_t\right)Q_0(\hat z^{N},z) &\le \tsum_{t=1}^{N-1}\gamma_tQ_0(z^{t+1},z) \\
&\le  {\cal B}_N(z,z^{[N]})-\frac{\gamma_{N-1}\eta_{N-1}}{2} \| x^N-x^{N-1}\|_2^2-\left\langle {Ax^{N}-Ax^{N-1}, y^{N}-y}\right\rangle\\
&\quad  +\tsum_{t=1}^{N-1}\left(\Delta'_t-\E_{i_t}[\Delta'_t]\right),
\end{array}
$$
which, in view of the fact that
\beq\label{Cauchy2}
-\left\langle {Ax^{N}-Ax^{N-1}, y^{N}-y}\right\rangle \le  \tfrac{\| A\|^2}{2\tau_{N-1}}\| x^N-x^{N-1}\|_2^2+\tfrac{\tau_{N-1}}{2}\| y^{N}-y\|_2^2,
\eeq
then implies that
$$
\begin{array}{lll}
\left(\tsum_{t=1}^{N-1}\gamma_t\right)Q_0(\hat z^{N},z)  &\le &{\cal B}_N(z,z^{[N]}) +\frac{\tau_{N-1}}{2}\| y^{N}-y\|_2^2-\left(\frac{\gamma_{N-1}\eta_{N-1}}{2}- \frac{\|A\|_2^2}{2\tau_{N-1}}\right) \| x^N-x^{N-1}\|_2^2 \\
&& +\tsum_{t=1}^{N-1}\left(\Delta'_t-\E_{i_t}[\Delta'_t]\right).
\end{array}
$$
Now it follows from \eqnok{stepsize_q2a}, \eqnok{stepsize_q2b}, and \eqnok{defB} that
$$
\begin{array}{ll}
&{\cal B}_N(z,z^{[N]})+\frac{\tau_{N-1}}{2}\| y^{N}-y\|_2^2 \\
&= \tfrac{\gamma_1\eta_1}{2}\| x-x^1\|_2^2-\tsum_{t=1}^{N-2} \left( \tfrac{\gamma_t\eta_t}{2}
-\tfrac{\gamma_{t+1}\eta_{t+1}}{2}\right)\| x-x^{t+1}\|_2^2 -\tfrac{\gamma_{N-1}\eta_{N-1}}{2}\| x-x^N\|_2^2\\
&\quad + \, \tfrac{\tau_1}{2}\| y-y^1\|_2^2-\tsum_{t=1}^{N-2} \left( \tfrac{\tau_t}{2}-\tfrac{\tau_{t+1}}{2}\right)\| y-y^{t+1}\|_2^2 \\
&\le \tfrac{\gamma_1\eta_1}{2}\| x-x^1\|_2^2- \tfrac{\gamma_{N-1}\eta_{N-1}}{2}\| x-x^N\|_2^2 +\tfrac{\tau_1}{2}\| y-y^1\|_2^2\\
&\le \tfrac{\gamma_1\eta_1}{2}\Omega_X^2 +\tfrac{\tau_1}{2}\Omega_Y^2.
\end{array}
$$
Combining the above two relations, and noting that $\frac{\gamma_{N-1}\eta_{N-1}}{2} \ge \frac{\|A\|_2^2}{2\tau_{N-1}}$ by \eqnok{stepsize_q3b}, we obtain
\beq\label{main_ieq}
\left(\tsum_{t=1}^{N-1}\gamma_t\right)Q_0(\hat z^{N},z) \le
\tfrac{\gamma_{1}\eta_{1}}{2}\Omega_X^2+\tfrac{\tau_{1}}{2}\Omega_Y^2  +\tsum_{t=1}^{N-1}\left(\Delta'_t-\E_{i_t}[\Delta'_t]\right).
\eeq
Taking expectation w.r.t $i_t, t=1,2,...,N-1,$ noting that $\E_{i_t} [\Delta'_t-\E_{i_t}[\Delta'_t]]= 0$ and $\tfrac{p-1}{p} \le 1,$ we obtain
$$
\begin{array}{ll}
\E_{[i_N]}[Q_0(\hat z^{N},z)] \le & \left(\tsum_{t=1}^{N-1}\gamma_t\right)^{-1} \left [\tfrac{\gamma_{N-1}\eta_{N-1}}{2}\Omega _X^2
+\tfrac{\tau_{N-1}}{2}\Omega _Y^2 \right ].
\end{array}
$$

The proof of part b) is similar to that of part a). The main idea is to break down the perturbation term $\Delta'_t$
into two parts, one independent on $y$ and the other depending on $y.$ More specifically, let us denote
\begin{align}
\Delta'_{t1}&=\left\langle {q_{t-1}U_{i_t} A(x^t-x^{t-1}), y^{t}}\right\rangle-\left\langle {{\bar U_{i_t}A}x^t, y^{t}}\right\rangle
+\tsum_{i\ne i_t} J_i(y_i^t)-\tfrac{\tau_t}{2}\| y^t-y^{t+1}\|_2^2, \label{def_Gamma_1}\\
\Delta'_{t2}&=\left\langle {q_{t-1}U_{i_t} A(x^t-x^{t-1}), y}\right\rangle-\left\langle {{\bar U_{i_t}A}x^t, y}\right\rangle+\tsum_{i\ne i_t} J_i(y_i).  \label{def_Gamma_2}
\end{align}
Clearly, we have
\beq \label{def_Gamma_12}
\Delta'_{t}=\Delta'_{t1}+\Delta'_{t2}.
\eeq
Using exactly the same analysis as in part a) except putting the perturbation term $\Delta'_{t2}$ to the left hand side of \eqnok{main_ieq}, we have
$$
\begin{array}{ll}
&\left(\tsum_{t=1}^{N-1}\gamma_t\right)Q_0(\hat z^{N},z) +\tsum_{t=1}^{N-1}\left( \Delta'_{t2}-\E_{i_t}[\Delta'_{t2}]\right)\le \tfrac{\gamma_{1}\eta_{1}}{2}\Omega_X^2+\tfrac{\tau_{1}}{2}\Omega_Y^2
+\tsum_{t=1}^{N-1}\left( \Delta'_{t1} -\E_{i_t}[\Delta'_{t1}]\right).
\end{array}
$$
Denoting
$$\delta(y)=\left(\tsum_{t=1}^{N-1}\gamma_t\right)^{-1}\tsum_{t=1}^{N-1}\left( \Delta'_{t2} -\E_{i_t}[\Delta'_{t2}]\right),$$
we then conclude from the above inequality that
$$
\left(\tsum_{t=1}^{N-1}\gamma_t\right)[Q_0(\hat z^{N},z) + \delta(y)]
 \le \tfrac{\gamma_{1}\eta_{1}}{2}\Omega_X^2+\tfrac{\tau_{1}}{2}\Omega_Y^2
+ \tsum_{t=1}^{N-1}\left( \Delta'_{t1} -\E_{i_t}[\Delta'_{t1}]\right).
$$
The result in \eqnok{cor4} then immediately follows by maximizing both sides of the above inequality w.r.t $z=(x,y)$, and
taking expectation w.r.t $i_t, t=1,2,...,N-1$, and using the definition of $g_\delta$ in \eqnok{def_gdelta}.
\end{proof}

\vgap

While there are many options to specify the parameters $\eta_t, \tau_t$, and $\gamma_t$
of the RPD method such that the assumptions in \eqnok{stepsize_q1a}-\eqnok{stepsize_q3b}
are satisfied,
below we provide a specific parameter setting which leads to an optimal rate of convergence
for the RPD algorithm in terms of its dependence on $N$.

\begin{corollary} \label{cor1}
Suppose that the initial point of Algorithm~\ref{algRPD} is set to $x^1=x^0$ and $y^1=\argmax_{y \in Y} \langle A x^1, y \rangle - J(y).$
Also assume that $q_t$ is set to \eqnok{stepsize_q1a}, and $\{\gamma_t\}$, $\{\tau_t\}$, and $\{\eta_t\}$ are set to
\begin{align}
\gamma_t&=\tfrac{1}{p}, \; t=1,2,...,N-2, \;\; \mbox{and} \;\; \gamma_{N-1}=1, \label{stepsize1_q1}\\
\tau_t&=\tfrac{\sqrt{p}\| A\|\Omega_X}{\Omega_Y},  \label{stepsize1_q2a}\\
\eta_t&=\tfrac{p^{\frac{3}{2}}\| A\|\Omega_Y}{\Omega_X}, t=1,2,...,N-2, \;\;
\mbox{and}\;\; \eta_{N-1}=\tfrac{\sqrt{p}\| A\|\Omega_Y}{\Omega_X}.  \label{stepsize1_q2b}
\end{align}
Then for any $N \ge 1,$ we have
\beq \label{convergence2}
\E[Q_0(\hat z^{N},z)] \le \tfrac{p^{3/2}\|A\|_2\Omega_X \Omega_Y}{N+p-2}, \; \forall z \in Z,
\eeq
Moreover, there exists a function $\delta(y)$ such that $\bbe[\delta(y)] = 0$ for any $y \in Y$ and
\beq \label{convergence2b}
\E[g_{\delta(y)}(\hat z^{N})] \le  \tfrac{p^{3/2}\|A\|_2\Omega_X \Omega_Y}{N+p-2}.
\eeq
\end{corollary}

\begin{proof}
It is easy to verify that $\gamma_t$, $\tau_t$, and $\eta_t$ defined in
\eqnok{stepsize1_q1}-\eqnok{stepsize1_q2b} satisfy \eqnok{stepsize_q1b}-\eqnok{stepsize_q3b}.
Moreover, it follows from \eqnok{stepsize1_q1}-\eqnok{stepsize1_q2b} that
$$\tsum_{t=1}^{N-1}\gamma_t=\tfrac{N+p-2}{p}, \ \
\tfrac{\gamma_{1}\eta_{1}}{2}\Omega_X^2=\tfrac{\sqrt{p}\|A\|_2\Omega_X\Omega_Y}{2} \ \
\mbox{and} \ \ \tfrac{\tau_{1}}{2}\Omega_Y^2=\tfrac{\sqrt{p} \|A\|_2\Omega_X\Omega_Y}{2}.$$
The results then follow by plugging these identities into \eqnok{convergence} and \eqnok{cor4}.
\end{proof}

\vgap

We now make some remarks about the convergence results  obtained in Theorem~\ref{theorem1} and
Corollary~\ref{cor1}. Observe that, in the view of \eqnok{convergence2}, the total number of iterations
required by the RPD algorithm to find an $\epsilon$-solution of problem~\eqnok{spp}, i.e.,
a point $\hat z \in Z$ such that $\E[Q_0(\hat z, z)] \le \epsilon$ for any $z \in Z$, can be bounded by
$
{\cal O}(p^{3/2}\|A\|_2\Omega_X \Omega_Y/\epsilon).
$
This bound is not improvable in terms of its dependence on $\epsilon$ for a given $p$ (see discussions in \cite{CheLanOu13-1}).
It should be noted, however, that the number of dual subprobems
to be solved in the RPD algorithm is larger than the one required by the deterministic primal-dual method, i.e.,
${\cal O}(p\|A\|_2\Omega_X \Omega_Y/\epsilon)$, by a factor of  $\sqrt{p}$.
On the other hand, in comparison with stochastic algorithms such as the stochastic mirror descent (SMD)
method (see \cite{NJLS09-1,Lan10-3,DangLan13-1}), Algorithm~\ref{algRPD} exhibits a significantly better dependence
on $\epsilon$, as the latter algorithm would require ${\cal O}(1/\epsilon^2)$ iterations to find an $\epsilon$-solution
of problem \eqnok{spp}-\eqnok{spp1}.

\subsection{Smooth bilinear saddle point problems over bounded feasible sets}
In this section, we assume that $J_i(y_i), i=1,2,\ldots,p$, in \eqnok{spp}-\eqnok{spp1} are strongly convex functions.
Moreover, without loss of generality we assume that their strong convexity modulus is given by $1$.
Under these assumptions, the objective function of \eqnok{spp} is a smooth convex function, which explains why these
problems are called smooth bilinear saddle point problems.
Our goal is to show that the RPD algorithm, when equipped with properly specified algorithmic parameters,
exhibits an optimal ${\cal O} (1/N^2)$ rate of convergence for solving this class of saddle point problems.

Similar to Proposition~\ref{proposition1}, we first establish an important recursion
for the RPD algorithm applied to smooth bilinear saddle point problems.
Note that this result involves an extra parameter $\theta_t$ in comparison with Proposition~\ref{proposition1}.

\begin{proposition} \label{proposition2}
Let $z^t = (x^t, y^t)$, $t = 1, \ldots, N$ be generated by the RPD algorithm. For any $z \in Z$, we have
\beq \label{recursion33}
\begin{array}{ll}
&\gamma_tQ_0(z^{t+1},z)+\left\langle {\gamma_tAx^{t+1}-\theta_tAx^t, y^{t+1}-y}\right\rangle+ (\theta_t-\gamma_t)\left[ J(y^{t+1})-J(y)\right]- \tilde \Delta_t \\
&\le \frac{\gamma_t\eta_t}{2}\left[\| x-x^t\|_2^2- \| x^t-x^{t+1}\|_2^2-\| x-x^{t+1}\|_2^2\right]+\frac{\theta_t \tau_t}{2}\| y-y^t\|_2^2\\
&\quad -\theta_t \frac{1+\tau_t}{2} \| y-y^{t+1}\|_2^2-\frac{\theta_t\tau_t}{2 }\| y^t-y^{t+1}\|_2^2
\end{array}
\eeq
for any $\theta_t \ge 0$,
where
\beq \label{def_t_delta}
\begin{array}{ll}
\tilde \Delta_t &:= \left\langle {q_{t-1}\theta_tU_{i_t} A(x^t-x^{t-1}), y^{t+1}-y}\right\rangle-\left\langle {\theta_t\bar U_{i_t}Ax^t, y^{t}-y}\right\rangle\\
&\;\;\;\;+\tsum_{i \ne i_t}\theta_t[J_j(y_{i}^{t})-J_i(y_i) + \tfrac{1}{2} \| y-y_i^{t+1}\|_2^2].
\end{array}
\eeq
\end{proposition}

\begin{proof}
It follows from the optimality condition of problem \eqnok{eqn_update_y} (e.g., Lemma 6 of \cite{LaLuMo11-1} and Lemma 2 of \cite{GhaLan12-2a})
and the strong convexity of $J_{i_t}(y_{i_t})$ (modulus $1$) that for all $y \in Y$,
\beq \label{observations5}
\begin{array}{ll}
&\left\langle {-U_{i_t}A{\bar x}^t, y^{t+1}-y}\right\rangle +J_{i_t}(y_{i_t}^{t+1})-J_{i_t}(y_{i_t}) \\
&\le \frac{\tau_t }{2}\| y_{i_t}-y_{i_t}^t\|_2^2-\frac{\tau_t}{2 } \| y_{i_t}^t-y_{i_t}^{t+1}\|_2^2-\frac{1+\tau_t}{2 }\| y_{i_t}-y_{i_t}^{t+1}\|_2^2.
\end{array}
\eeq
This relation, in view of the observations in \eqnok{observation111} and \eqnok{observation112}, then implies that
\beq\label{observation3a}
\begin{array}{ll}
&\left\langle {-Ax^t, y^{t+1}-y}\right\rangle-\left\langle {q_{t-1}U_{i_t} A(x^t-x^{t-1}), y^{t+1}-y}\right\rangle+\left\langle {\bar U_{i_t}Ax^t, y^{t}-y}\right\rangle+J (y^{t+1})-J(y)\\
&\le \tsum_{i \ne i_t}[J_i (y_{i}^{t})-J_i(y_i)] +\tfrac{\tau_t}{2}\left[\| y-y^t\|_2^2-\| y-y^{t+1}\|_2^2-\|y^t-y^{t+1} \|_2^2\right]-\tfrac{1}{2}\| y_{i_t}-y_{i_t}^{t+1}\|_2^2\\
&= \tsum_{i \ne i_t} [J_i (y_{i}^{t})-J_i(y_i)]+  \frac{\tau_t}{2}\| y-y^t\|_2^2-\tfrac{1+\tau_t}{2 } \| y-y^{t+1}\|_2^2
 -\frac{\tau_t}{2 }\| y^t-y^{t+1}\|_2^2 +\tfrac{1}{2} \tsum_{i \ne i_t}^p \| y_i-y_i^{t+1}\|_2^2,
\end{array}
\eeq
Multiplying both sides of the above inequality by $\theta_t$ and
both sides of \eqnok{observation3} by $\gamma_t$, and then adding them up, we obtain
\[
\begin{array}{ll}
&\gamma_tQ_0(z^{t+1},z)+\left\langle {\gamma_tAx^{t+1}-\theta_tAx^t, y^{t+1}-y}\right\rangle+ (\theta_t-\gamma_t)\left[ J(y^{t+1})-J(y)\right]\\
&-\left\langle {q_{t-1}\theta_tU_{i_t} A(x^t-x^{t-1}), y^{t+1}-y}\right\rangle+\left\langle {\theta_t\bar U_{i_t}Ax^t, y^{t}-y}\right\rangle\\
&\le  \tsum_{i \ne i_t}\theta_t[J_i(y_{i}^{t})-J_i(y_i)] + \frac{\gamma_t\eta_t}{2}\left[\| x-x^t\|_2^2- \| x^t-x^{t+1}\|_2^2-\| x-x^{t+1}\|_2^2\right]
+ \frac{\theta_t \tau_t}{2}\| y-y^t\|_2^2\\
&-\theta_t\left(\frac{1+\tau_t}{2 }\right) \| y-y^{t+1}\|_2^2-\frac{\theta_t\tau_t}{2 }\| y^t-y^{t+1}\|_2^2
+\tsum_{i \ne i_t} \tfrac{\theta_t}{2}\| y_i-y_i^{t+1}\|_2^2,
\end{array}
\]
which, in view of the definition of $\tilde \Delta_t$ in \eqnok{def_t_delta}, then implies \eqnok{recursion33}.
\end{proof}

\vgap

The following lemma provides an upper bound on $\E_{i_t}[\tilde\Delta_t].$
\begin{lemma} \label{lemma2_delta}
Let $\tilde \Delta_t$ be defined in \eqnok{def_t_delta}.
If $i_t$ is uniformly distributed on $\{1,2,...,p\},$ then
$$\begin{array}{ll}
\E_{i_t}[\tilde\Delta _t] \le& \left\langle{\left( \tfrac{1}{p}q_{t-1}\theta_t-\tfrac{p-1}{p}\theta_t \right) Ax^t-\tfrac{1}{p}q_{t-1}\theta_tAx^{t-1}, y^{t}-y}\right\rangle
+\tfrac{p-1}{p} \theta_t\left[ J (y^{t})-J(y)\right]\\
&+\tfrac{q_{t-1}^2 \theta_t\|A\|_2^2}{2p\tau_t}\| x^t-x^{t-1}\|_2^2+\tfrac{\tau_t\theta_t}{2}\E_{i_t} \left[\| y^{t+1}-y^t\|_2^2\right]+\frac{p-1}{2p}\theta_t\| y-y^{t}\|_2^2.
\end{array}
$$
\end{lemma}

\begin{proof}
By the definition of $\tilde\Delta  _t$ in \eqnok{def_t_delta}, we have
$$\begin{array}{ll}
\tilde\Delta  _t=&\left\langle {q_{t-1}\theta_tU_{i_t} A(x^t-x^{t-1})-\theta_t\bar U_{i_t}Ax^t, y^{t}-y}\right\rangle
-\left\langle {q_{t-1}\theta_t U_{i_t} A(x^t-x^{t-1}), y^{t+1}-y^t}\right\rangle\\
&+\tsum_{i \ne i_t}\theta_t[J_i (y_{i}^{t})-J_i(y_i) + \tfrac{1}{2} \| y_i-y_i^{t+1}\|_2^2],
\end{array}
$$
The result then immediately follows from the above identity,
the relations \eqnok{recur_expectation}, \eqnok{norm_equal}, \eqnok{cauchy}, and the facts that $\theta_t \ge 0$ and
\beq \label{norm_equal2}
\E_{i_t} \left[ \tsum_{i \ne i_t}\theta_t\| y_i-y_i^{t+1}\|_2^2\right]=\tfrac{p-1}{p}\theta_t\| y-y^{t}\|_2^2.
\eeq
\end{proof}

\vgap

We are now ready to establish the main convergence properties
of the RPD algorithm applied to smooth bilinear saddle point problems.

\begin{theorem} \label{theorem2}
Suppose that the initial point of Algorithm~\ref{algRPD} is chosen such that $x^1=x^0$ and $y^1=\argmax_{y \in Y} \langle A x^1, y \rangle - J(y).$
Also assume that the parameters $q_t$ and the weights $\gamma_t, \theta_t$ are set to
\begin{align}
\theta_t&= \tfrac{1}{p}q_t \theta_{t+1}, \; i=1,\ldots,N-1, \label{stepsize_stronglyconvex_q1a}\\
\gamma_{t}&=(\tfrac{1}{p}q_t-\tfrac{p-1}{p})\theta_{t+1},t=1,\ldots,N-2, \label{stepsize_stronglyconvex_q1b}\\
\gamma_{N-1}&=\theta_{N-1}, \label{stepsize_stronglyconvex_q2a}\\
\theta_t\left(1+ \tau_{t}\right) &\ge \theta_{t+1}\left(\tfrac{p-1}{p}+ \tau_{t+1}\right), \; i=1,\ldots,N-1, \label{stepsize_stronglyconvex_q2b}\\
\gamma_t\eta_{t}&\ge\gamma_{t+1}\eta_{t+1} \; i=1,\ldots,N-1, \label{stepsize_stronglyconvex_q3a}\\
p\gamma_{t-1}\eta_{t-1}\tau_t&\ge q_{t-1}^2\theta_t\|A\|_2^2, \; i=1,\ldots,N-1, \label{stepsize_stronglyconvex_q3b}\\
\eta_{N-1} \tau_{N-1}&\ge \|A\|_2^2. \label{stepsize_stronglyconvex_q3c}
\end{align}
\begin{itemize}
\item [a)] For any $N \ge 1,$ we have
\beq \label{convergence_stronglyconvex}
\E[Q_0(\hat z^{N},z)] \le  \left(\tsum_{t=1}^{N-1}\gamma_t\right)^{-1} \left [\tfrac{\gamma_1\eta_{1}}{2}
\Omega _X^2+\theta_1\left(\tfrac{p-1}{p}+\tau_1\right)\Omega _Y^2 \right ],
\eeq
where the expectation is taken w.r.t. to $i_{[N]}=(i_1,...,i_{N-1})$.
\item [b)] For any $N \ge 1,$ there exists a function $\delta(y)$ such that $\bbe[\delta(y)] = 0$ for any $y \in Y$ and
\beq \label{convergence_stronglyconvex1}
\E[g_{\delta(y)}(\hat z^N)] \le
 \left( \tsum_{t=1}^{N-1}\gamma_t\right)^{-1} \left [\tfrac{\gamma_{1}\eta_{1}}{2}\Omega _X^2+\theta_1\left(\tfrac{p-1}{p}+\tau_1\right)\Omega _Y^2 \right ].
\eeq
\end{itemize}
\end{theorem}

\begin{proof}
We first show part a). It follows from
Proposition~\ref{proposition2} and Lemma~\ref{lemma2_delta} that
\[
\begin{array}{ll}
&\gamma_t Q_0(z^{t+1},z)+\left\langle {\gamma_tAx^{t+1}-\theta_tAx^t, y^{t+1}-y}\right\rangle+ (\theta_t-\gamma_t)\left[ J(y^{t+1})-J(y)\right ] \\
&\le \left\langle{\left( \tfrac{1}{p}q_{t-1}-\tfrac{p-1}{p} \right) \theta_t Ax^t-\tfrac{1}{p}q_{t-1}\theta_tAx^{t-1}, y^{t}-y}\right\rangle
+\tfrac{p-1}{p} \theta_t\left[ J (y^{t})-J(y)\right]
+ \tilde \Delta_t-\E_{i_t}[\tilde \Delta_t]\\
&\quad + \,\tfrac{q_{t-1}^2 \theta_t\|A\|_2^2}{2p\tau_t}\| x^t-x^{t-1}\|_2^2+\tfrac{\theta_t \tau_t}{2}\E_{i_t} \left[\| y^{t+1}-y^t\|_2^2\right]
+\theta_t \left(\tfrac{p-1}{2p}+\tfrac{\tau_t}{2}\right)\| y-y^t\|_2^2\\
&\quad + \, \tfrac{\gamma_t\eta_t}{2}\left[\| x-x^t\|_2^2- \| x^t-x^{t+1}\|_2^2-\| x-x^{t+1}\|_2^2\right]
-\theta_t\left(\tfrac{1+\tau_t}{2 }\right) \| y-y^{t+1}\|_2^2-\tfrac{\theta_t\tau_t}{2 }\| y^t-y^{t+1}\|_2^2 .
\end{array}
\]
Denoting
$$\tilde\Delta'_t =\tilde\Delta_t -\tfrac{\theta_t\tau_t}{2}\| y^t-y^{t+1}\|_2^2,$$
we can rewrite the above inequality as
\beq \label{recursion1_new2_delta}
\begin{array}{ll}
&\gamma_tQ_0(z^{t+1},z)+\left\langle {\gamma_tAx^{t+1}-\theta_tAx^t, y^{t+1}-y}\right\rangle+ (\theta_t-\gamma_t)\left[ J(y^{t+1})-J(y)\right]\\
&\le \left\langle{\left( \tfrac{1}{p}q_{t-1}-\tfrac{p-1}{p} \right) \theta_t Ax^t-\tfrac{1}{p}q_{t-1}\theta_tAx^{t-1}, y^{t}-y}\right\rangle+\tfrac{p-1}{p} \theta_t\left[ J (y^{t})-J(y)\right]\\
&\quad + \, \tfrac{q_{t-1}^2 \theta_t\|A\|_2^2}{2p\tau_t}\| x^t-x^{t-1}\|_2^2+\tfrac{\gamma_t\eta_t}{2}\left[\| x-x^t\|_2^2- \| x^t-x^{t+1}\|_2^2-\| x-x^{t+1}\|_2^2\right]\\
&\quad + \, \theta_t \left(\tfrac{p-1}{2p}+\tfrac{\tau_t}{2}\right)\| y-y^t\|_2^2-\theta_t \tfrac{1+\tau_t}{2 } \| y-y^{t+1}\|_2^2+ \tilde \Delta'_t-\E_{i_t}[\tilde \Delta'_t].
\end{array}
\eeq
Observe that by \eqnok{stepsize_stronglyconvex_q1a}, \eqnok{stepsize_stronglyconvex_q1b}, and \eqnok{stepsize_stronglyconvex_q2a},
and the fact $x^1 = x^0$,
\begin{align*}
\langle {\gamma_tAx^{t+1}-\theta_tAx^t, y^{t+1}-y}\rangle &=
\langle{( \tfrac{1}{p}q_{t}-\tfrac{p-1}{p} ) \theta_{t+1} Ax^{t+1}-\tfrac{1}{p}q_{t}\theta_{t+1} Ax^{t}, y^{t+1}-y}\rangle\\
\langle{( \tfrac{1}{p}q_{0}-\tfrac{p-1}{p} ) \theta_1 Ax^1-\tfrac{1}{p}q_{0}\theta_1Ax^{0}, y^{1}-y}\rangle
&= \langle{( -\tfrac{p-1}{p} ) \theta_1 Ax^1, y^{1}-y}\rangle\\
(\theta_t-\gamma_t)\left[ J(y^{t+1})-J(y)\right] &= \tfrac{p-1}{p} \theta_{t+1}\left[ J (y^{t+1})-J(y)\right]\\
(\theta_{N-1}-\gamma_{N-1})\left[ J(y^{N})-J(y)\right] &= 0.
\end{align*}
Taking summation from $t=1$ to $N-1$ on both sides of \eqnok{recursion1_new2_delta}, using the above observations,
 and
denoting
\beq \label{defB_stronglyconvex}
\begin{array}{lll}
{\cal {\tilde B}}_N(z,z^{[N]}) &:=& \tsum_{t=1}^{N-1} \left[ \tfrac{\gamma_t\eta_t}{2}\| x-x^t\|_2^2-\tfrac{\gamma_t\eta_t}{2}\| x-x^{t+1}\|_2^2\right]\\
&&\quad +\, \tsum_{t=1}^{N-1} \left[ \tfrac{\theta_t}{2} \left(\tfrac{p-1}{p}+\tau_t\right)\| y-y^t\|_2^2-\tfrac{\theta_t}{2} (1+\tau_t)\| y-y^{t+1}\|_2^2\right].
\end{array}
\eeq
we obtain
\beq \label{summation_new_delta}
\begin{array}{lll}
\tsum_{t=1}^{N-1}\gamma_tQ_0(z^{t+1},z) &\le& \tilde{\cal B}_N(z,z^{[N]})  -\left\langle{\gamma_{N-1}Ax^{N}-\theta_{N-1}Ax^{N-1}, y^{N}-y}\right\rangle
-\tfrac{p-1}{p}\theta_1 \left\langle{Ax^1, y^{1}-y}\right\rangle\\
&&\quad +\tfrac{p-1}{p} \theta_1\left[ J (y^{1})-J(y)\right]
- \tsum_{t=1}^{N-2}\left(\tfrac{\gamma_t\eta_t}{2} -\tfrac{q_{t}^2\theta_{t+1}\|A\|_2^2}{2p\tau_{t+1} }\right) \| x^{t+1}-x^{t}\|_2^2 \\
&&\quad -\tfrac{\gamma_{N-1}\eta_{N-1}}{2} \| x^N-x^{N-1}\|_2^2 +\tsum_{t=1}^{N-1}\left(\tilde \Delta'_t - \E_{i_t}[\tilde \Delta'_t]\right)\\
 &\le& \tilde{\cal B}_N(z,z^{[N]})- \gamma_{N-1} \left\langle{Ax^{N}-Ax^{N-1}, y^{N}-y}\right\rangle-\tfrac{\gamma_{N-1}\eta_{N-1}}{2} \| x^N-x^{N-1}\|_2^2\\
&&\quad +\tsum_{t=1}^{N-1}\left(\tilde \Delta'_t - \E_{i_t}[\tilde \Delta'_t]\right),
\end{array}
\eeq
where the second inequality follows from the facts that $\gamma_{N-1}=\theta_{N-1}$ and $y^1=\argmax_{y \in Y} \langle A x^1, y \rangle - J(y)$,
and relation \eqnok{stepsize_stronglyconvex_q3b}.
The above conclusion, in view of the definition of $\hat x^N$ and the convexity of $Q_0(\hat z, z)$ in terms of $\hat z$,
then implies that
$$
\begin{array}{lll}
\left(\tsum_{t=1}^{N-1}\gamma_t\right)Q_0(\hat z^{N},z) &\le&
\tsum_{t=1}^{N-1}\gamma_tQ_0(z^{t+1},z) \le
 \tilde{\cal B}_N(z,z^{[N]})-\tfrac{\gamma_{N-1}\eta_{N-1}}{2} \| x^N-x^{N-1}\|_2^2\\
&&-\gamma_{N-1}\left\langle {Ax^{N}-Ax^{N-1}, y^{N}-y}\right\rangle +\tsum_{t=1}^{N-1}\left(\tilde \Delta'_t - \E_{i_t}[\tilde \Delta'_t] \right).
\end{array}
$$
Now by the Cauchy-Swartz inequality and the reltation that $\gamma_{N-1}=\theta_{N-1}$ in \eqnok{stepsize_stronglyconvex_q2a},
\beq\label{Cauchy2_new}
-\gamma_{N-1}\left\langle {Ax^{N}-Ax^{N-1}, y^{N}-y}\right\rangle \le  \tfrac{\gamma_{N-1}\| A\|^2}{2\tau_{N-1}}\| x^N-x^{N-1}\|_2^2+\tfrac{\theta_{N-1}\tau_{N-1}}{2}\| y^{N}-y\|_2^2.
\eeq
Moreover, by \eqnok{stepsize_stronglyconvex_q2b} and \eqnok{stepsize_stronglyconvex_q3a}, we have
$$
\begin{array}{lll}
{\cal {\tilde B}}_N(z,z^{[N]}) &=& \tfrac{\gamma_1\eta_1}{2}\| x-x^1\|_2^2-\tsum_{t=1}^{N-2} \left( \tfrac{\gamma_t\eta_t}{2}-\tfrac{\gamma_{t+1}\eta_{t+1}}{2}\right)\| x-x^{t+1}\|_2^2 -\tfrac{\gamma_{N-1}\eta_{N-1}}{2}\| x-x^N\|_2^2\\
&& + \, \tfrac{\theta_1}{2} \left(\tfrac{p-1}{p}+\tau_1\right)\| y-y^1\|_2^2
-\tsum_{t=1}^{N-2} \left[\tfrac{\theta_t}{2} \left(1+ \tau_t\right)-\tfrac{\theta_{t+1}}{2} \left(\tfrac{p-1}{p}+\tau_{t+1}\right)\right]\| y-y^{t+1}\|_2^2 \\
&&\;\;\;\;-\tfrac{\theta_{N-1}}{2} \left(1+\tau_{N-1}\right)\| y^{N}-y\|_2^2\\
&\le& \tfrac{\gamma_1\eta_1}{2}\| x-x^1\|_2^2+\tfrac{\theta_1}{2} \left(\tfrac{p-1}{p}+\tau_1\right)\| y-y^1\|^2-\tfrac{\theta_{N-1}\tau_{N-1}}{2}\| y^{N}-y\|_2^2\\
&\le& \tfrac{\gamma_1\eta_{1}}{2}\Omega_X^2+\tfrac{\theta_1}{2} \left(\tfrac{p-1}{p}+\tau_1\right)\Omega_Y^2-\tfrac{\theta_{N-1}\tau_{N-1}}{2}\| y^{N}-y\|_2^2.
\end{array}
$$
Combining the above three relations, and noting that $\eta_{N-1}\tau_{N-1} \ge \|A\|_2^2$ by \eqnok{stepsize_stronglyconvex_q3c}, we obtain
$$
\left(\tsum_{t=1}^{N-1}\gamma_t\right)Q_0(\hat z^{N},z) \le \tfrac{\gamma_1\eta_{1}}{2}\Omega_X^2+\theta_1\left(\tfrac{p-1}{p}+\tfrac{\tau_1}{2}\right)\Omega_Y^2
+ \tsum_{t=1}^{N-1}(\tilde\Delta'_t -\E_{i_t} [\tilde \Delta'_t]).
$$
Taking expectation w.r.t $i_t, t=1,2,...,N-1,$ and noting that $\E_{i_t} [\tilde\Delta'_t -\E_{i_t} [\tilde \Delta'_t]]= 0$ and $\tfrac{p-1}{p} \le 1,$ we obtain the result in part a).
The proof of part b) is similar to that in Theorem~\ref{theorem1} and hence the details are skipped.
\end{proof}

While there are many options to specify the parameters $\eta_t, \tau_t$, $\theta_t$ and $\gamma_t$
of the RPD method such that the assumptions in \eqnok{stepsize_stronglyconvex_q1a}-\eqnok{stepsize_stronglyconvex_q3c}
are satisfied,
below we provide a specific parameter setting which leads to an optimal rate of convergence
for the RPD algorithm in terms of its dependence on $N$.
\begin{corollary} \label{cor2}
Suppose that the initial point of Algorithm~\ref{algRPD} is set to $x^1=x^0$ and $y^1=\argmax_{y \in Y} \langle A x^1, y \rangle - J(y).$
Also assume that $q_t, \gamma_t, \theta_t, \tau_t$ and $\eta_t$ are set to
%
 \begin{align}
q_t&=p\tfrac{t+3p}{t+3p+1}, \; i=1,\ldots,N-1, \label{stepsize_stronglyconvex_q1a_cor}\\
\gamma_t&=\tfrac{t+2p+1}{p}, \; t=1,2,\ldots,N-2 \; \mbox{and} \; \gamma_{N-1}=N+3p-1, \label{stepsize_stronglyconvex_q1b_cor}\\
\theta_t&=t+3p, \; t=1,2,\ldots,N-1, \label{stepsize_stronglyconvex_q2a_cor}\\
\tau_t&=\tfrac{t+p}{2p} \; i=1,\ldots,N-1, \label{stepsize_stronglyconvex_q2b_cor}\\
\eta_t &=\tfrac{2p^3\|A\|_2^2}{t+2p+1}. \label{stepsize_stronglyconvex_q3c_cor}
\end{align}
Then, for any $N \ge 1$ we have
\beq \label{convergence2_stronglyconvex}
\begin{array}{ll}
\E_{[i_N]}[Q_0(\hat z^{N},z)] \le & \tfrac{2}{N(N+p)}\left [p^3 \|A\|_2^2\Omega _X^2+4.5p^2\Omega_Y^2 \right ].
\end{array}
\eeq
Moreover, there exists a function $\delta(y)$ such that $\bbe[\delta(y)] = 0$ for any $y \in Y$ and
\beq \label{convergence_stronglyconvex2}
\E[g_{\delta(y)}(\hat z^N)] \le
\tfrac{2}{N(N+p)} \left [p^3\| A\|_2^2\Omega _X^2+4.5p^2\Omega_Y^2 \right ].
\eeq
\end{corollary}
\begin{proof}
It is easy to verify that $q_t, \theta_t, \gamma_t, \tau_t$ and $\eta_t$ defined in \eqnok{stepsize_stronglyconvex_q1a_cor}-\eqnok{stepsize_stronglyconvex_q3c_cor}  satisfy \eqnok{stepsize_stronglyconvex_q1a}-\eqnok{stepsize_stronglyconvex_q3c}.
Moreover, it follows from \eqnok{stepsize_stronglyconvex_q1a_cor}-\eqnok{stepsize_stronglyconvex_q3c_cor} that
$$\tsum_{t=1}^{N-1}\gamma_t=N+3p-1+\tsum_{t=1}^{N-2}\tfrac{t+2p+1}{p} \ge \tfrac{N(N+p)}{2p},$$
$$\tfrac{\gamma_1\eta_{1}}{2}=p^2\|A\|_2^2 \;\;\mbox{and} \;\;\theta_1\left(\tfrac{p-1}{p}+\tau_{1}\right)=\tfrac{(1+3p)(3p-1)}{2p}.$$
The results then follow by plugging these identities into \eqnok{convergence_stronglyconvex} and \eqnok{convergence_stronglyconvex1} and using the fact that $\tfrac{(1+3p)(3p-1)}{2p} \le 4.5p.$
\end{proof}

\vgap

We now make some remarks about the convergence results obtained in Theorem~\ref{theorem2} and Corollary~\ref{cor2}. Observe that, in view of \eqnok{convergence2_stronglyconvex}, the total number of iterations required by the RPD algorithm to find an $\epsilon$-solution of
smooth bilinear saddle point problems, i.e.,
a point $\hat z \in Z$ such that $\E[Q_0(\hat z, z)] \le \epsilon$ for any $z \in Z$, can be bounded by
\[
\max \left\{\tfrac{\sqrt{2}p^{\tfrac{3}{2}}\|A\|_2\Omega_X}{\sqrt{\epsilon}},\tfrac{3p\Omega_Y}{\sqrt{\epsilon}}\right\}.
\]
Similar to the previous results for general bilinear saddle point problems, this bound is not improvable in terms of its dependence on $\epsilon$ for a given $p.$

\section{Generalization of the randomized primal-dual method}
In this section, we discuss two possible ways to generalize the RPD method. One is to
extend it for solving unbounded saddle point problems and the other is to incorporate non-Euclidean distances.

\subsection{RPD for unbounded saddle point problems} \label{unbounded_sec}

In this subsection, we assume that either the primal feasible set $X$ or dual feasible set $Y$ is unbounded.
To assess the quality of a feasible solution $\hat z \in X \times Y,$ we use the perturbation-based criterion defined in \eqnok{def_tgdelta}.
Throughout this subsection we assume that both $h$ and $J$ are general convex function (without assuming strong convexity)
so that problems \eqnok{eqn_update_y} and \eqnok{eqn_update_x} are relatively easy to solve. Our goal is to show that the RPD algorithm,
when equipped with properly specified algorithmic parameters, exhibits an ${\cal O}(1/N)$ rate of convergence
for solving this class of unbounded saddle point problems.

Before establishing the main convergence properties for the RPD algorithm applied to unbounded bilinear saddle point problems,
we first show an important property of the RPD method which states that, for every $t \le N,$ the expected distance from $z^t$ to
a given saddle point $z^*$ is bounded.
\begin{lemma}
Let $z^t=(x^t,y^t), t=1,2,...,N$, be generated by the Algorithm~\ref{algRPD} with $x^1=x^0$ and $y^1=\argmax_{y \in Y} \langle A x^1, y \rangle - J(y).$
Also assume that $q_t, \tau_t, \eta_t$ and $\gamma_t$ are set to \eqnok{stepsize_q1a}-\eqnok{stepsize_q3b}.
If
\beq \label{stepsize_unbounded}
\begin{array}{lll}
\tau_{t-1}&=&\tau_{t}, \; i=1,..,N-1,\\
\gamma_{t-1}\eta_{t-1}&=& \gamma_{t}\eta_{t}, \; i=1,..,N-1,\\
\gamma_{t-1}&=& \gamma_{t}, \; i=1,..,N-2,\\
\end{array}
\eeq
then we have
\begin{align}
  \E_{[i_t]}\left[\| x^*-x^{t}\|_2^2 \right] &\le 2D^2, \;\; \forall t \le N-1, \label{bound1}\\
  \E_{[i_t]}\left[\| y^*-y^{t}\|_2^2 \right] &\le \tfrac{2(2-\gamma_{t-1})\eta_{t-1}}{\tau_{t-1}}D^2, \;\; \forall t \le N-1, \label{bound3}
\end{align}
and
\begin{align}
  \E_{[i_N]}\left[\| x^*-x^{N}\|_2^2 \right] &\le D^2, \label{bound2}\\
  \E_{[i_N]}\left[\| y^*-y^{N}\|_2^2 \right] &\le \tfrac{\eta_{N-1}\gamma_{N-1}}{\tau_{N-1}}D^2, \label{bound4}
\end{align}
where $[i_t]=\{ i_1,...,i_{t-1}\}$,
\beq \label{def_D}
D:=\sqrt{\| x^*-x^1\|_2^2+\tfrac{\tau_{1}}{\eta_{1}\gamma_{1}}\| y^*-y^1\|_2^2},
\eeq
and $z^*=(x^*,y^*)$ is a saddle point of problem~\eqnok{spp}.
\end{lemma}

\begin{proof}
We first prove \eqnok{bound2} and \eqnok{bound4}. Using \eqnok{summation} (with $z = z^*$), \eqnok{stepsize_unbounded},
and the fact that $x^1=x^0$ and $y^1=\argmax_{y \in Y} \langle A x^1, y \rangle - J(y),$ we obtain
$$
\begin{array}{lll}
\tsum_{t=1}^{N-1}\gamma_tQ_0(z^{t+1},z^*) &\le& \tfrac{\gamma_1\eta_1}{2}\| x^1-x^*\|_2^2-\tfrac{\gamma_{N-1}\eta_{N-1}}{2}\| x^N-x^*\|_2^2+\tfrac{\tau_1}{2}\| y^1-y^*\|_2^2-\tfrac{\tau_{N-1}}{2}\| y^N-y^*\|_2^2\\
&& \,+\tsum_{t=1}^{N-1}\left( \Delta'_t-\E_{i_t}[\Delta'_t]\right).
\end{array}
$$
Taking expectation on both sides of the above inequality w.r.t $[i_N],$ noting that $Q_0(z^{t+1},z^*) \ge 0, \forall t \ge 1$ and $\E_{i_t}[\Delta'_t-\E_{i_t}[\Delta'_t]]=0,$ then we obtain \eqnok{bound2} and \eqnok{bound4}.

Now let us show \eqnok{bound1} and \eqnok{bound3}. Using similar analysis to \eqnok{recursion1_new2},
we can show that for any $j \le t$ and any $t \le N-1,$
$$\begin{array}{ll}
&\gamma_jQ_0(z^{j+1},z)+\left\langle {\gamma_jAx^{j+1}-Ax^j, y^{j+1}-y}\right\rangle+ (1-\gamma_j)\left[ J(y^{j+1})-J(y)\right] \\
& \le\left\langle{\left( \tfrac{1}{p}q_{j-1}-\tfrac{p-1}{p} \right) Ax^j-\tfrac{1}{p}q_{j-1}Ax^{j-1}, y^{j}-y}\right\rangle+\tfrac{p-1}{p} \left[ J (y^{j})-J(y)\right]-\tfrac{\gamma_t\eta_t}{2}\| x^j-x^{j+1}\|_2^2\\
&+\tfrac{q_{j-1}^2 \|A\|_2^2}{2p\tau_j}\| x^j-x^{j-1}\|_2^2+\tfrac{\gamma_j\eta_j}{2}\left[\| x-x^j\|_2^2-\| x-x^{j+1}\|_2^2\right]+\tfrac{\tau_j}{2}\left[\| y-y^j\|_2^2-\| y-y^{j+1}\|_2^2\right]+\Delta'_j-\E_{i_j}[\Delta'_j].
\end{array}
$$
Taking summation on both sides of the above inequality from $j=1$ to $t-1$ and using the facts that $x^1=x^0, y^1=\argmax_{y \in Y} \langle A x^1, y \rangle - J(y)$ and $\gamma_j=\tfrac{1}{p}, \forall j=1,2,...,t-1,$ we have
\beq \label{summation_new}
\begin{array}{ll}
&\tsum_{t=j}^{t-1}\gamma_jQ_0(z^{j+1},z) + (1-\gamma_{t-1})\left[ J(y^{t})-J(y)\right]\\
& \quad \le {\cal B}_t(z,z^{[t]})-\tfrac{\gamma_{t-1}\eta_{t-1}}{2} \| x^t-x^{t-1}\|_2^2-\left\langle {\gamma_{t-1}Ax^{t}-Ax^{t-1}, y^{t}-y}\right\rangle+\tsum_{j=1}^{t-1}\left(\Delta'_j-\E_{i_j}[\Delta'_j]\right).
\end{array}
\eeq
Observe that
\beq \label{observation_u2}
\begin{array}{ll}
-\left\langle {\gamma_{t-1}Ax^{t}-Ax^{t-1}, y^{t}-y}\right\rangle &= -\left\langle {Ax^{t}-Ax^{t-1}, y^{t}-y}\right\rangle+\left\langle {(1-\gamma_{t-1})Ax^{t}, y^{t}-y}\right\rangle\\
&= (1-\gamma_{t-1}) [\left\langle {Ax, y^{t}}\right\rangle-\left\langle {Ax^{t}, y}\right\rangle+\left\langle {Ax^{t}, y^{t}}\right\rangle-\left\langle {Ax, y^{t}}\right\rangle]\\
&\;\;\;\;-\left\langle {Ax^{t}-Ax^{t-1}, y^{t}-y}\right\rangle,
\end{array}
\eeq
which, in view of the fact that by the optimality condition of problem~\eqnok{eqn_update_x},
$$\left\langle {Ax^{t}-Ax, y^{t}}\right\rangle \le \tfrac{\eta_{t-1}}{2}[\| x-x^{t-1}\|_2^2-\|x-x^{t}\|_2^2-\|x^{t}-x^{t-1}\|_2^2]-[h(x^t)-h(x)],$$
then implies that
$$
\begin{array}{ll}
-\left\langle {\gamma_{t-1}Ax^{t}-Ax^{t-1}, y^{t}-y}\right\rangle
\le & -\left\langle {Ax^{t}-Ax^{t-1}, y^{t}-y}\right\rangle+(1-\gamma_{t-1}) [\left\langle {Ax, y^{t}}\right\rangle-\left\langle {Ax^{t}, y}\right\rangle-h(x^t)+h(x)]\\
&+\tfrac{\eta_{t-1}(1-\gamma_{t-1})}{2}[\| x-x^{t-1}\|_2^2-\|x-x^{t}\|_2^2-\|x^{t}-x^{t-1}\|_2^2].
\end{array}
$$
By the above inequality and \eqnok{summation_new}, we have
\[
\begin{array}{ll}
&\tsum_{t=j}^{t-1}\gamma_jQ_0(z^{j+1},z) + (1-\gamma_{t-1}) \left[\left\langle {Ax^{t}, y}\right\rangle-\left\langle {Ax, y^{t}}\right\rangle+h(x^t)-h(x)+ J(y^{t})-J(y)\right]\\
\le & {\cal B}_t(z,z^{[t]})-\tfrac{\gamma_{t-1}\eta_{t-1}}{2} \| x^t-x^{t-1}\|_2^2-\left\langle {Ax^{t}-Ax^{t-1}, y^{t}-y}\right\rangle\\
&+\tfrac{\eta_{t-1}(1-\gamma_{t-1})}{2}[\| x-x^{t-1}\|_2^2-\|x-x^{t}\|_2^2-\|x^{t}-x^{t-1}\|_2^2]+\tsum_{j=1}^{t-1}\left(\Delta'_j-\E_{i_j}[\Delta'_j]\right).
\end{array}
\]
Using the previous relation, the fact that
$$
\begin{array}{ll}
-\left\langle {Ax^{t}-Ax^{t-1}, y^{t}-y}\right\rangle & \le\| A\| \| x^{t}-x^{t-1}\|_2 \| y^{t}-y\|_2 \\
 &\le \tfrac{\|A\|_2^2}{\tau_{t-1}}\| x^t-x^{t-1}\|_2^2+\tfrac{\tau_{t-1}}{4}\| y^{t}-y\|_2^2.
\end{array}
$$
and the definition of $Q_0$ (with $z = z^*$), we conclude that
$$
\begin{array}{ll}
&\tsum_{j=1}^{t-1}\tilde\gamma_j Q_0(z^{j+1},z^*)\\
\le &\tfrac{\gamma_1\eta_1}{2}\| x^*-x^1\|_2^2-\left(\tfrac{\gamma_{t-1}\eta_{t-1}}{2}+ \tfrac{(1-\gamma_{t-1})\eta_{t-1}}{2}\right) \| x^*-x^{t}\|_2^2 \\
&-\left(\tfrac{\gamma_{t-1}\eta_{t-1}}{2}+ \tfrac{(1-\gamma_{t-1})\eta_{t-1}}{2}-\tfrac{\|A\|_2^2}{\tau_{t-1}}\right) \| x^t-x^{t-1}\|_2^2 +\tfrac{\tau_1}{2}\| y^*-y^1\|_2^2-\left(\tfrac{\tau_{t-1}}{2}- \tfrac{\tau_{t-1}}{4}\right) \| y^*-y^{t}\|_2^2\\
&+ \tfrac{(1-\gamma_{t-1})\eta_{t-1}}{2}\| \hat x-x^{t-1}\|_2^2+\tsum_{j=1}^{t-1}\left( \Delta'_j-E_{i_j}[\Delta'_j]\right),
\end{array}
$$
where $\tilde \gamma_j=\gamma_j, j=1,...,t-2$ and $\tilde \gamma_{t-1}=1.$

Now by the definition of gap function, we know that $Q_0(z^{j+1},z^*) \ge 0 \; \forall j \ge 1.$ Taking expectation
on both sides of the above inequality w.r.t $[i_t],$ noting that $\E_{i_j}[\Delta'_j-\E_{i_j}[\Delta'_j]]=0,$ $\gamma_1\eta_1=\gamma_{t-1}\eta_{t-1},$ $\tau_1=\tau_{t-1}$ and $\tfrac{\gamma_{t-1}\eta_{t-1}}{2}+ \tfrac{(1-\gamma_{t-1})\eta_{t-1}}{2}-\tfrac{\|A\|_2^2}{\tau_{t-1}} \ge 0$, we have
$$
 \begin{array}{ll}
&\left(\tfrac{\gamma_{t-1}\eta_{t-1}}{2}+ \tfrac{(1-\gamma_{t-1})\eta_{t-1}}{2}\right) \E_{[i_t]}\left[\| x^*-x^{t}\|_2^2 \right] + \tfrac{\tau_{t-1}}{4} \E_{[i_t]}\left[\| y^*-y^{t}\|_2^2 \right]\\
 \le & \tfrac{\gamma_{t-1}\eta_{t-1}}{2}\| x^*-x^1\|_2^2+\tfrac{\tau_{t-1}}{2}\| y^*-y^1\|_2^2+\tfrac{(1-\gamma_{t-1})\eta_{t-1}}{2}\E_{[i_{t-1}]}\left[\| x^*-x^{t-1}\|_2^2 \right].
\end{array}
$$
Dividing both sides of the above relation by $\gamma_{t-1}\eta_{t-1}/2$, we obtain
 \beq \label{summation4}
 \begin{array}{ll}
&\left(1+ \tfrac{1-\gamma_{t-1}}{\gamma_{t-1}}\right) \E_{[i_t]}\left[\| x^*-x^{t}\|_2^2 \right] + \tfrac{\tau_{t-1}}{2\eta_{t-1}\gamma_{t-1}} \E_{[i_t]}\left[\| y^*-y^{t}\|_2^2 \right]\\
 \le & \| x^*-x^1\|_2^2+\tfrac{\tau_{t-1}}{\eta_{t-1}\gamma_{t-1}}\|  y^*-y^1\|_2^2+\tfrac{1-\gamma_{t-1}}{\gamma_{t-1}}\E_{[i_{t-1}]}\left[\| x^*-x^{t-1}\|_2^2 \right],
\end{array}
\eeq
which implies that
$$\begin{array}{ll}
\left(1+ \tfrac{1-\gamma_{t-1}}{\gamma_{t-1}}\right) \E_{[i_t]}\left[\| x^*-x^{t}\|_2^2 \right] \le \| x^*-x^1\|_2^2+\tfrac{\tau_{t-1}}{\eta_{t-1}\gamma_{t-1}}\| y^*-y^1\|_2^2+\tfrac{1-\gamma_{t-1}}{\gamma_{t-1}}\E_{[i_{t-1}]}\left[\| x^*-x^{t-1}\|_2^2 \right].
\end{array}
$$
For simplicity, let us denote
$$a=\tfrac{1-\gamma_{i-1}}{\gamma_{i-1}}, i=1,...,t.$$
Then, using the previous inequality, and the definition of $D$, we have for any $t \le N,$
$$
 \begin{array}{ll}
\E_{[i_t]}\left[\| x^*-x^{t}\|_2^2 \right] &\le \tfrac{D^2}{1+a}+\tfrac{a}{1+a}\E_{[i_{t-1}]}\left[\| x^*-x^{t-1}\|_2^2 \right]\\
& \le \tfrac{D^2}{1+a}+\tfrac{a}{1+a}\left( \tfrac{D^2}{1+a}+\tfrac{a}{1+a}\E_{[i_{t-2}]}\left[\| x^*-x^{t-2}\|_2^2 \right]\right)\\
& = \tfrac{D^2}{1+a} \left( 1+\tfrac{a}{1+a}\right)+\tfrac{a^2}{(1+a)^2}\E_{[i_{t-2}]}\left[\| x^*-x^{t-2}\|_2^2 \right]\\
& \le ... \\
& \le \tfrac{D^2}{1+a} \left( 1+\tfrac{a}{1+a}+...+\tfrac{a^{t-2}}{(1+a)^{t-2}}\right)+\tfrac{a^{t-1}}{(1+a)^{t-1}}\| x^*-x^{1}\|_2^2\\
& = \tfrac{D^2}{1+a} \tfrac{1-(\tfrac{a}{1+a})^{t-1}}{1-\tfrac{a}{1+a}}+\tfrac{a^{t-1}}{(1+a)^{t-1}}\| x^*-x^{1}\|_2^2  \le 2D^2.
\end{array}
$$
Plugging the above inequality to \eqnok{summation4}, for any $t \le N-1,$ we obtain
$$
\tfrac{\tau_{t-1}}{2\eta_{t-1}\gamma_{t-1}} \E_{[i_t]}\left[\| y^*-y^{t}\|_2^2 \right] \le D^2+\tfrac{1-\gamma_{t-1}}{\gamma_{t-1}}2D^2,
$$
which implies \eqnok{bound3}.
\end{proof}

\vgap

The following result provide an important bound on $\E_{[i_N]} \left[Q_0(\hat z^{N},z)\right]$ for the unbounded saddle point problems.
\begin{lemma}
Let $z^t=(x^t,y^t), t=1,2,...,N$ be generated by the Algorithm~\ref{algRPD} with $x^1=x^0$ and $y^1=\argmax_{y \in Y} \langle A x^1, y \rangle - J(y).$ Also assume that $q_t, \tau_t, \eta_t$ and $\gamma_t$ are set to \eqnok{stepsize_q1a}-\eqnok{stepsize_q3b} and \eqnok{stepsize_unbounded}. The there exists a vector $v_N$ such that
\beq \label{eq7}
\E_{[i_N]} \left[Q_0(\hat z^{N},z) + \left\langle {v_N,\hat z^N -z} \right\rangle \right]
\le \left( \tsum_{t=1}^{N-1}\gamma_t\right)^{-1} \left[ \tfrac{\gamma_{N-1}\eta_{N-1}}{2}\E_{[i_N]}  \left[\| \hat x^N-x^1\|_2^2\right]+\tfrac{\tau_{N-1}}{2}\E_{[i_N]}  \left[\| \hat y^N-y^1\|_2^2\right] \right].
\eeq
\end{lemma}
\begin{proof}
First note that
$$
 \begin{array}{lll}
\|x-x^{1}\|_2^2-\|x-x^{N}\|_2^2 & =& 2\left\langle {x^N-x^1,x} \right\rangle + \|x^{1}\|_2^2-\|x^{N}\|_2^2\\
& =&2\left\langle {x^N-x^1,x-\hat x^N} \right\rangle + 2\left\langle {x^N-x^1,\hat x^N} \right\rangle + \|x^{1}\|_2^2-\|x^{N}\|_2^2\\
& = &2\left\langle {x^N-x^1,x-\hat x^N} \right\rangle + \|x^{1}-\hat x^N\|_2^2-\|x^{N}-\hat x^N\|_2^2.
\end{array}
$$
Using this identity in \eqnok{summation} and the fact that $x^1=x^0$ and $y^1=\argmax_{y \in Y} \langle A x^1, y \rangle - J(y),$ we obtain
\beq \label{eq11}
\begin{array}{lll}
&&\tsum_{t=1}^{N-1}\gamma_t Q_0(z^{t+1},z) +  \left\langle {Ax^N-Ax^{N-1}, \hat y^N-y} \right\rangle \\
&&+ \tfrac{\gamma_{N-1}\eta_{N-1}}{2}\left\langle {x^N-x^{1},\hat x^N-x} \right\rangle+\tfrac{\tau_{N-1}}{2}\left\langle {y^N-y^{1},\hat y^N-y} \right\rangle\\
&\le &\tfrac{\gamma_{N-1}\eta_{N-1}}{2}\| \hat x^N-x^1\|_2^2-\tfrac{\gamma_{N-1}\eta_{N-1}}{2} \| \hat x^N-x^{N}\|_2^2 -\tfrac{\gamma_{N-1}\eta_{N-1}}{2}\| x^N-x^{N-1}\|_2^2 +\tfrac{\tau_{N-1}}{2}\| \hat y^N-y^1\|_2^2\\
&&-\tfrac{\tau_{N-1}}{2} \| \hat y^N-y^{N}\|_2^2-\left\langle {Ax^{N}-Ax^{N-1}, y^{N}-\hat y^N}\right\rangle+\tsum_{t=1}^{N-1}\left( \Delta'_t-\E_{i_t}[\Delta'_t]\right).
\end{array}
\eeq
Denoting
\beq \label{defV}
v_N=\left( \tsum_{t=1}^{N-1}\gamma_t\right)^{-1}\left(\tfrac{ \gamma_{N-1}\eta_{N-1}}{2}(x^N-x^1),( Ax^{N}-Ax^{N-1})+\tfrac{\tau_{N-1}}{2} (y^N-y^1) \right),
\eeq
and using the fact that $Q_0(z^{t+1},z)$ is linear, we conclude from \eqnok{eq11} that
\[
\begin{array}{ll}
&\left( \tsum_{t=1}^{N-1}\gamma_t\right) \left[Q_0(\hat z^{N},z) + \langle {v_N,\hat z^N -z} \rangle \right] \\
\le &\tfrac{\gamma_{N-1}\eta_{N-1}}{2}\| \hat x^N-x^1\|_2^2-\tfrac{\gamma_{N-1}\eta_{N-1}}{2} \| \hat x^N-x^{N}\|_2^2 -\tfrac{\gamma_{N-1}\eta_{N-1}}{2} \| x^N-x^{N-1}\|_2^2 +\tfrac{\tau_{N-1}}{2} \| \hat y^N-y^1\|_2^2\\
&-\tfrac{\tau_{N-1}}{2}  \| \hat y^N-y^{N}\|_2^2-\langle {Ax^{N}-Ax^{N-1}, y^{N}-\hat y^N}\rangle+\tsum_{t=1}^{N-1}\left( \Delta'_t-\E_{i_t}[\Delta'_t]\right),
\end{array}
\]
which together with the facts that
$$
\begin{array}{lll}
-\left\langle {Ax^{N}-Ax^{N-1}, y^{N}-\hat y^N}\right\rangle & \le& \|A\|_2 \| x^{N}-x^{N-1}\|_2 \| y^{N}-\hat y^N\|_2 \\
 &\le& \tfrac{\|A\|_2^2}{2\tau_{N-1}}\| x^N-x^{N-1}\|_2^2+\tfrac{\tau_{N-1}}{2} \| y^{N}-\hat y^N\|_2^2,
\end{array}
$$
and $\gamma_{N-1} \eta_{N-1}\tau_{N-1}\ge \|A\|_2^2,$ then impy that
$$
\begin{array}{ll}
Q_0(\hat z^{N},z) + \left\langle {v_N,\hat z^N -z} \right\rangle \le \left( \tsum_{t=1}^{N-1}\gamma_t\right)^{-1} \left[\tfrac{\gamma_{N-1}\eta_{N-1}}{2}\| \hat x^N-x^1\|_2^2+\tfrac{\tau_{N-1}}{2}\| \hat y^N-y^1\|_2^2+\tsum_{t=1}^{N-1}\left( \Delta'_t-\E_{i_t}[\Delta'_t]\right) \right].
\end{array}
$$
The result now immediately follows by taking expectation w.r.t $[i_N]$ on the both
 sides of the above inequality and noting that $\E_{i_t} [\Delta'_t-\E_{i_t}[\Delta'_t]]=0.$
\end{proof}

\vgap

The following theorem shows that the rate of convergence of the RPD algorithm for solving the unbounded saddle point problems.
\begin{theorem} \label{unbouded_theorem}
Let $z^t=(x^t,y^t), t=1,2,...,N$ be generated by Algorithm~\ref{algRPD} with $x^1=x^0$ and $y^1=\argmax_{y \in Y} \langle A x^1, y \rangle - J(y).$ Also assume that $q_t, \tau_t, \eta_t$ and $\gamma_t$ are set to \eqnok{stepsize_q1a}-\eqnok{stepsize_q3b} and \eqnok{stepsize_unbounded}.
\begin{itemize}
\item [a)]
For any $N \ge 1,$ there exists a vector $v_N$ such that
\begin{align}
\E_{[i_N]} \left[Q_0(\hat z^{N},z) + \left\langle {v_N,\hat z^N -z} \right\rangle \right]
&\le \tfrac{[3\gamma_{N-1}\eta_{N-1}+ 2(2-\gamma_{N-1})\eta_{N-1}]D^2}{\tsum_{t=1}^{N-1}\gamma_t}, \label{eq17}\\
\E _{[i_N]}\left[ \|v_N\|_2\right] &\le \tfrac{KD}{\tsum_{t=1}^{N-1}\gamma_t}, \label{vn_bound}
\end{align}
where
$$K=2\gamma_{N-1}\eta_{N-1}+2\sqrt{\gamma_{N-1}\tau_{N-1}\eta_{N-1}}+\|A\|_2 (1 + \sqrt{2}).$$
\item [b)]
For any $N \ge 1,$ there exists a vector $\sigma(y)$ such that \eqnok{vn_bound} holds, $\E[\sigma(y)]=0$ for any $y \in Y,$ and
\beq \label{eq117}
\E_{[i_N]} \left[\tilde g_{\sigma(y)}(\hat z^{N},z,v_N) \right]
\le \tfrac{[3\gamma_{N-1}\eta_{N-1}+ 2(2-\gamma_{N-1})\eta_{N-1}]D^2}{\tsum_{t=1}^{N-1}\gamma_t}.
\eeq
\end{itemize}

\end{theorem}
\begin{proof}
We firs show the part a). It follows from \eqnok{bound1}, \eqnok{bound3}, \eqnok{bound2}, and \eqnok{bound4} that
$$
\begin{array}{lll}
  &&  \E_{[i_N]} \left[ \| \left(\gamma_{N-1}\eta_{N-1}(x^N-x^1), Ax^{N}-Ax^{N-1}+\tau_{N-1}(y^N-y^1)\right) \|_2\right] \\
&\le&   \E_{[i_N]} \left[  \| Ax^N-Ax^{N-1}\|_2+\gamma_{N-1}\eta_{N-1}\| x^N-x^1\|_2+\tau_{N-1}\| y^N-y^1\|_2\right]\\
&\le &\E_{[i_N]}\left[\|A\|_2 \| x^N-x^{N-1}\|_2+\gamma_{N-1}\eta_{N-1}\| x^N-x^1\|_2+\tau_{N-1}\| y^N-y^1\|_2\right]\\
&\le & \E_{[i_N]} \left[\gamma_{N-1}\eta_{N-1}\left(\| x^N-x^*\|_2+\| x^1-x^*\|_2\right)+\tau_{N-1}\left(\| y^N-y^*\|_2+\| y^1-y^*\|_2\right) \right]\\
& &+\E_{[i_N]} \left[\|A\|_2(\| x^N-x^*\|_2 + \| x^{N-1} - x^*\|_2)\right]\\
&\le& 2\gamma_{N-1}\eta_{N-1}D+2\sqrt{\gamma_{N-1}\tau_{N-1}\eta_{N-1}}D+\|A\|_2 (D + \sqrt{2}D)=KD.
\end{array}
$$
The above inequality and the definition of $v_N$ imply \eqnok{vn_bound}.
On the other hand, using \eqnok{bound1}, \eqnok{bound3}, \eqnok{bound2},  and \eqnok{bound4}, we have
$$
\begin{array}{lll}
& &\tfrac{\gamma_{N-1}\eta_{N-1}}{2}\E_{[i_N]}  \left[\| \hat x^N-x^1\|_2^2\right]+\tfrac{\tau_{N-1}}{2}\E_{[i_N]}  \left[\| \hat y^N-y^1\|_2^2\right] \\
&\le&\gamma_{N-1}\eta_{N-1}\E_{[i_N]} \left[ \| \hat x^N-x^*\|_2^2+\| x^*-x^1\|_2^2\right]+\tau_{N-1}\E_{[i_N]} \left[\| \hat y^N-y^*\|_2^2+\| y^*-y^1\|_2^2\right]\\
&=&  \gamma_{N-1}\eta_{N-1}D^2+\E_{[i_N]} \left[ \gamma_{N-1}\eta_{N-1}\| \hat x^N-x^*\|_2^2+ \tau_{N-1}\|\hat y^N-y^*\|_2^2\right]\\
&\le&   \gamma_{N-1}\eta_{N-1}D^2+\tfrac{1}{\tsum_{t=1}^{N-1}\gamma_t}\E _{[i_N]}\left[ \tsum_{t=1}^{N-1}\gamma_t \left(\gamma_{t-1}\eta_{t-1}\|x^t-x^*\|_2^2+\tau_{t-1} \| y^t-y^*\|_2^2 \right)\right]\\
&=& \gamma_{N-1}\eta_{N-1}D^2+\tfrac{1}{\tsum_{t=1}^{N-1}\gamma_t} \left[ \tsum_{t=1}^{N-1}\gamma_t \left(\gamma_{t-1}\eta_{t-1}\E_{[i_t]}\left[\|x^t-x^*\|_2^2\right]+ \tau_{t-1}\E_{[i_t]}\left[\| y^t- y^*\|_2^2\right] \right)\right]\\
&\le&  \gamma_{N-1}\eta_{N-1}D^2+\tfrac{1}{\tsum_{t=1}^{N-1}\gamma_t} \left[ \tsum_{t=1}^{N-1}\gamma_t \left(2\gamma_{t-1}\eta_{t-1}D^2+ \tau_{t-1}\tfrac{2(2-\gamma_{t-1})\eta_{t-1}}{\tau_{t-1}}D^2 \right)\right] \\
&=& \left(3\gamma_{N-1}\eta_{N-1}+ 2(2-\gamma_{N-1})\eta_{N-1}\right)D^2 .
\end{array}
$$
Using the above inequality in \eqnok{eq7}, we obtain \eqnok{eq17}.
The proof of part b) is similar to that of Theorem~\ref{theorem1}.b) and hence the details are skipped.
\end{proof}

\vgap

Below we specify a parameter setting that satisfies the assumptions in \eqnok{stepsize_q1a}-\eqnok{stepsize_q3b} and \eqnok{stepsize_unbounded} and leads to an optimal rate of convergence for the RPD algorithm in terms of its dependence on $N$.
\begin{corollary} \label{unboundded_cor}
Let $z^t=(x^t,y^t), t=1,2,...,N$ be generated by Algorithm~\ref{algRPD} with $x^1=x^0$ and $y^1=\argmax_{y \in Y} \langle A x^1, y \rangle - J(y).$ Also assume that $\gamma_t, q_t, \tau_t$ and $\eta_t$ are set to
\begin{align}
  q_{t} &= p,  \;\; \forall t \le N, \label{stepsize_unbounded1}\\
  \gamma_t&=\tfrac{1}{p}, \forall t=1,2,..,N-2, \;\;\mbox{and} \;\; \gamma_{N-1}=1, \label{stepsize_unbounded2}\\
  \tau_t&=\|A\|_2p^{3/2}, \; i=1,..,N-1, \label{stepsize_unbounded3}\\
  \eta_t&=\|A\|_2p^{3/2}, \; i=1,..,N-2, \;\; \mbox{and} \;\; \eta_{N-1}=\|A\|_2p^{1/2}. \label{stepsize_unbounded4}
\end{align}
Then for any $N \ge 1,$ there exists a vector $v_N$ such that
\begin{align}
\E_{[i_N]} \left[Q_0(\hat z^{N},z) + \left\langle {v_N,\hat z^N -z} \right\rangle \right]
&\le \tfrac{5p^{3/2}\|A\|_2D^2}{N+p-2}, \label{eq18} \\
\E_{[i_N]} \left[ \|v_N\|_2\right]
&\le \tfrac{p}{N+p-2}\left(4p^{1/2}+ (1+\sqrt{2})\right)\|A\|_2D. \label{eq19}
\end{align}
Moreover, for any $N \ge 1,$ there exists a vector $\sigma(y)$ such that \eqnok{eq19} holds, $\E[\sigma(y)]=0$ for any $y \in Y,$ and
\beq \label{eq118}
\E_{[i_N]} \left[\tilde g_{\sigma(y)}(\hat z^{N},z,v_N) \right]
\le \tfrac{5p^{3/2}\|A\|_2D^2}{N+p-2}.
\eeq
\end{corollary}
\begin{proof}It is easy to verify that $\gamma_t, q_t, \tau_t$ and $\eta_t$ defined in \eqnok{stepsize_unbounded1}-\eqnok{stepsize_unbounded4} satisfy \eqnok{stepsize_q1a}-\eqnok{stepsize_q3b} and \eqnok{stepsize_unbounded}.
We also have
$$\tsum_{t=1}^{N-1}\gamma_t=\tfrac{(N+p-2)}{p}.$$
Plugging this identity into \eqnok{eq17}-\eqnok{eq117}, we obtain \eqnok{eq18}-\eqnok{eq118} respectively.
\end{proof}


\vgap

A few remarks about the results obtained in Theorem~\ref{unbouded_theorem} and
Corollary~\ref{unboundded_cor} are in place. First, in the view of \eqnok{eq18}, the total number of iterations
required by the RPD algorithm to find an $\epsilon$-solution of problem~\eqnok{spp}, i.e.,
a point $\hat z \in Z$ such that $\E[Q_0(\hat z, z)+ \left\langle {v_N,\hat z^N -z} \right\rangle] \le \epsilon$ for any $z \in Z$, can be bounded by
$
{\cal O}(p^{3/2}\|A\|_2\Omega_X \Omega_Y/\epsilon).
$
Second, similar to the bounded problems, these results are new and optimal in terms of its dependence on $\epsilon$ for a given $p$ (see discussions in \cite{CheLanOu13-1}).
To the best of our knowledge, this is the first time such an optimal rate of convergence is obtained
in the literature for a randomized algorithm for solving the saddle point problem \eqnok{spp}-\eqnok{spp1} with unbounded domains.

\subsection{Non-Euclidean randomized primal-dual methods}
In this subsection, we show that
by replacing the usual Euclidean distance by generalized non-Euclidean prox-functions, Algorithm~\ref{algRPD} can be
adaptive to different geometry of the feasible sets .

Recall that a function $\omega_i:Y_i \rightarrow R$ is a distance generating function~\cite{NJLS09-1}
with modulus $\alpha_i$ with respect to $\| \cdot \| _i$, if $\omega_i$ is continuously
differentiable and strongly convex with parameter $\alpha_i$
with respect to $\| \cdot \|_i$.
Without loss of generality, we assume that $\alpha_i=1$ for any $i = 1, \ldots, b$,
because we can always rescale $\omega_i(y) $ to $\bar \omega_i(y) =\omega_i(y) / \alpha_i$ in case $\alpha_i \ne 1$. Therefore, we have
$$
\langle y  - z, \nabla \omega_i ( y ) - \nabla \omega_i { (z)} \rangle
\ge \| {x  - z} \|_i^2 \ \ \forall y ,z \in Y_i .
$$
The prox-function associated with $\omega_i$ is given by
\beq \label{def_vi}
V_i (z,y) = \omega _i (y) - [\omega _i (z)
+ \langle {\nabla \omega_i (z),y - z } \rangle ] \ \ \forall y, z \in Y_i.
\eeq
The prox-function $V_i(\cdot,\cdot)$ is also called the Bregman's distance,
which was initially studied by Bregman \cite{Breg67}.
Suppose that the set $Y_i$ is bounded, the distance generating function
$\w_i$ also gives rise to the diameter of $Y_i$, which
will be used frequently in our convergence analysis:
\beq \label{def_D_i}
{\cal D}_{\omega _i ,Y_i }
:= {\mathop {\max }\limits_{y \in Y_i } \omega _i (y)
- \mathop {\min }\limits_{y \in Y_i } \omega _i (y)} .
\eeq
For the sake of notational convenience, sometimes we simply denote
${\cal D}_{\omega _i ,Y_i }$ by ${\cal D}_i$,
 $V(y,z)=\tsum_{i=1}^p V_i(y^{(i)},z^{(i)})$, $\forall y,z \in Y$, and $D_Y=\tsum_{i=1}^p D_i.$
Let $y_1^{(i)} = \argmin_{y \in Y_i} \w_i(y)$, $i =1, \ldots, b$. We can easily see that
for any $y \in Y$,
\beq \label{bound_V_i}
\begin{array}{ll}
\|y_1^{(i)} - y^{(i)}\|_i^2 / 2 &\le V_i (y_1^{(i)} ,y^{(i)} ) = \w_i(y^{(i)})
- \w_i(y_1^{(i)}) - \langle \nabla
\w_i(y_1^{(i)}), y^{(i)}- y_1^{(i)}\rangle\\
&\le \w_i(y^{(i)}) - \w_i(y_1^{(i)}) \le {\cal D}_i.
\end{array}
\eeq
Moreover, we define
$\| y\|^2 = \| y^{(1)}\|_{1}^2+\ldots+\| y^{(p)}\|_{p}^2$
and denote its conjugate by
$\| y\|_*^2 = \| y^{(1)}\|_{1,*}^2+\ldots+\| y^{(p)}\|_{p,*}^2$.
Similarly, letting $\omega: X \to \bbr$ be continuously
differentiable and strongly convex w.r.t $\|\cdot\|$ with modulus $1$,
we define the prox-function $V(\cdot,\cdot)$ associated with $\omega$ and use $D_X$ to denote the diameter of $X.$

We are now ready to describe a non-Euclidean variant of Algorithm~\ref{algRPD},
which is obtained by replacing the Euclidean distances used in the two subproblems \eqnok{eqn_update_y} and \eqnok{eqn_update_x}
in Step 2 of Algorithm~\ref{algRPD} with the Bregman's distances in \eqnok{eqn_update_y_g} and \eqnok{eqn_update_x_g}.

\begin{algorithm} [H]
    \caption{The non-Euclidean RPD Method}
    \label{algGRPD}
    \begin{algorithmic}
\STATE Let $z^1=(x^1,y^1) \in X\times Y$,
and nonnegative stepsizes $\{\tau_t\},$ $\{\eta_t\},$ parameters $\{q_t\}$, and weights $\{\gamma_t\}$
be given. Set $\bar x^1 = x^1$.
\FOR {$t=1, \ldots,N$}

\STATE 1. Generate a random variable $i_t$ uniformly from $\{1,2,...,p \}.$

\STATE 2. Update $y^{t+1}$ and $x^{t+1}$ by

\begin{align}
y_i^{t+1} &=
\begin{cases}
\argmin_{y_{i_t} \in Y_{i_t}}{\left\langle {-U_{i_t}A\bar x^t,y}\right\rangle+J_{i_t} (y_{i_t})+\tau_t V_{i_t}(y_{i_t},y_{i_t}^t) }, & i = i_t,\\
y_i^{t}, & i \ne i_t.
\end{cases}  \label{eqn_update_y_g}\\
x^{t+1}&=\argmin_{x \in X} {h(x)+\left\langle {x,A^Ty^{t+1}}\right\rangle}+\eta_t V(x,x^t). \label{eqn_update_x_g}\\
\bar x^{t+1}&=q_t (x^{t+1}-x^t)+ x^{t+1}. \label{eqn_x_bar_g}
\end{align}

\ENDFOR

{\bf Output:}  Set \beq \label{defxbar_g}\hat z^N=\left(\tsum_{t=1}^{N-1}{\gamma_t}\right)^{-1}\tsum_{t=1}^{N-1}{\gamma_t} z^{t+1}.\eeq

    \end{algorithmic}
\end{algorithm}

We will show that the non-Euclidean RPD algorithms exhibit similar convergence properties to the
Euclidean RPD algorithm for solving general bilinear saddle point problems with bounded feasible sets,
but they can be more flexible on the selection of the norms and distance generating functions.

First, the following result generalizes Proposition~\ref{proposition1}.
\begin{proposition} \label{proposition1_g}
Let $\{y^t\}_{t \ge 1}$ and $\{x^t\}_{t \ge 1}$ be generated by Algorithm~\ref{algGRPD}. Then for any $z \in Z,$ we have
\beq \label{recursion1_g}
\begin{array}{ll}
&\gamma_tQ_0(z^{t+1},z)+\left\langle {\gamma_tAx^{t+1}-Ax^t, y^{t+1}-y}\right\rangle+ (\gamma_t-1)\left[ J(y)-J(y^{t+1})\right]-\Delta _t \\
&\le \gamma_t\eta_t\left[V(x,x^t)- V(x^t,x^{t+1})-V(x,x^{t+1})\right]+\tau_t\left[V(y,y^t)-V(y,y^{t+1})-V(y^t,y^{t+1})\right],
\end{array}
\eeq
where $\Delta _t$ is defined in \eqnok{def_triangle_t}.
\end{proposition}

\begin{proof}
By the optimality condition of problem \eqnok{eqn_update_x_g}, for all $x \in X,$ we have
\beq \label{observation2_g}
h(x^{t+1})-h(x)+\left\langle {x^{t+1}-x, A^Ty^{t+1}}\right\rangle
+\eta_tV(x^t,x^{t+1})+\eta_tV(x,x^{t+1}) \le \eta_tV( x,x^t).
\eeq
Similarly, by the optimality condition of problem \eqnok{eqn_update_y_g}, for all $y \in Y,$ we have
\beq \label{observation1_g}
\left\langle {-U_{i_t}A{\bar x}^t, y^{t+1}-y}\right\rangle +J_{i_t}(y_{i_t}^{t+1})-J_{i_t}(y_{i_t})
+\tau_tV_{i_t}(y_{i_t}^t,y_{i_t}^{t+1})+\tau_tV_{i_t}(y_{i_t},y_{i_t}^{t+1}) \le \tau_tV_{i_t}(y_{i_t},y_{i_t}^t).
\eeq
The result follows by using an argument similar to the one used in the proof of Proposition~\ref{proposition1} by replacing
the Euclidean distances with Bregman's distances and noting that
$$
\begin{array}{lll}
V_{i_t}(y_{i_t}^t,y_{i_t}^{t+1})&=&V(y^t,y^{t+1}),\\
V_{i_t}y_{i_t},y_{i_t}^t)-V_{i_t}(y_{i_t},y_{i_t}^{t+1})&=&V(y,y^t)-V(y,y^{t+1}).
\end{array}
$$
\end{proof}

The following lemma provides an upper bound of $\E_{i_t}[\Delta_t].$
\begin{lemma} \label{lemma1_g}
If $i_t$ is uniformly distributed on $\{1,2,...,p\},$ then
$$\begin{array}{ll}
\E_{i_t}[\Delta _t] \le& \left\langle{\left( \tfrac{1}{p}q_{t-1}-\tfrac{p-1}{p} \right) Ax^t-\tfrac{1}{p}q_{t-1}Ax^{t-1}, y^{t}-y}\right\rangle+\tfrac{p-1}{p} \left[ J (y^{t})-J(y)\right]\\
&+\tfrac{q_{t-1}^2 \|A\|^2}{2p\tau_t}\| x^t-x^{t-1}\|^2+\tfrac{\tau_t}{2}\E_{i_t} \left[\| y^{t+1}-y^t\|^2\right].
\end{array}
$$
\end{lemma}

\begin{proof}
The proof is similar to the one of Lemma~\ref{lemma1} except that
we now replace the Euclidean distances by Bregman's distances, and that in \eqnok{cauchy}, we use the fact that $\| x-z\|^2 /2\le V(x,z),$ i.e.,
$$\begin{array}{lll}
\E_{i_t} \left[\left\langle {q_{t-1}U_{i_t}A(x^t-x^{t-1}),y^{t+1}-y^t}\right\rangle \right]&\le& \E_{i_t} \left[q_{t-1} \| U_{i_t}A( x^t-x^{t-1})\| \| y^{t+1}-y^t\|\right]\\
&\le& \E_{i_t} \left[\tfrac{q_{t-1}^2}{2\tau_t}\| U_{i_t}A(x^t-x^{t-1})\|^2+\tfrac{\tau_t}{2}\| y^{t+1}-y^t\|^2\right]\\
&\le& \tfrac{q_{t-1}^2}{2\tau_t}\E_{i_t} \left[\| U_{i_t}A(x^t-x^{t-1})\|^2\right]+\tfrac{\tau_t}{2}\E_{i_t} \left[\| y^{t+1}-y^t\|^2\right]\\
&= & \tfrac{q_{t-1}^2 }{2p\tau_t}\| A(x^t-x^{t-1})\|^2+\tfrac{\tau_t}{2}\E_{i_t} \left[\| y^{t+1}-y^t\|^2\right]\\
&\le& \tfrac{q_{t-1}^2 \|A\|^2}{2p\tau_t}\| x^t-x^{t-1}\|^2+\tfrac{\tau_t}{2}\E_{i_t} \left[\| y^{t+1}-y^t\|^2\right]\\
&\le& \tfrac{q_{t-1}^2 \|A\|^2}{p\tau_t}V(x^t,x^{t-1})+\tau_t\E_{i_t} \left[V(y^{t+1},y^t)\right].
\end{array}
$$
\end{proof}

Theorem~\ref{theorem1_g} below describes some convergence properties of
the non-Euclidean RPD methods.

\begin{theorem} \label{theorem1_g}
Suppose that the starting point $z^1$ is chosen such that $x^1=x^0$ and $y^1=\argmax_{y \in Y} \langle A x^1, y \rangle - J(y).$ Also assume that the parameters $q_t, \gamma_t, \tau_t$ and $\eta_t$ are set to \eqnok{stepsize_q1a}-\eqnok{stepsize_q3b}.
Then, for any $N \ge 1,$ we have
\beq \label{convergence_g}
\begin{array}{ll}
\E_{[i_N]}[Q_0(\hat z^{N},z)] \le & \left( \tsum_{t=1}^{N-1}\gamma_t\right)^{-1} \left [\tfrac{\gamma_{1}\eta_{1}}{2}D_X+\tfrac{\tau_{1}}{2}D_Y \right ], \; \forall z \in Z .
\end{array}
\eeq
where $\hat z^{N}$ is defined in \eqnok{defxbar} and the expectation is taken w.r.t. to $i_{[N]}=(i_1,...,i_{N-1})$.
\end{theorem}
\begin{proof}
The proof is almost identical to that of Theorem~\ref{theorem1} except that
we replace the Euclidean distances with Bregman's distances and that in \eqnok{Cauchy2} we use the fact $V(x,x_1) \le D_X$ and $V(y,y_1) \le D_Y,$ i.e.,
$$
\begin{array} {lll}
-\left\langle {Ax^{N}-Ax^{N-1}, y^{N}-y}\right\rangle &\le& \tfrac{\| A\|^2}{2\tau_{N-1}}\| x^N-x^{N-1}\|^2+\tfrac{\tau_{N-1}}{2}\| y^{N}-y\|^2\\
 &\le& \tfrac{\| A\|^2}{\tau_{N-1}}V(x^N,x^{N-1})+\tau_{N-1}V(y^{N},y).
 \end{array}
$$
\end{proof}

\vgap

It should be noted that we can also establish the convergence of the generalized algorithm for smooth bilinear saddle point problems.
However, it is still not clear to us whether Algorithm~\ref{algGRPD} can be generalized to the case when neither $X$ nor $Y$ are bounded.

\section{RPD for linearly constrained problems and its relation to ADMM}
Our goal of this section is to show that Algorithm~\ref{algRPD} applied to the linearly constrained optimization
problems can be viewed exactly as a randomized proximal alternating direction of multiplier method (ADMM).

More specifically, consider the following optimization problem
\beq \label{LCP}
\begin{array}{ll}
\min & f_1(x_1)+f_2(x_2)+...+f_p(x_p) \\
s.t.& A_1x_1+A_2x_2+...+A_px_p=b
\end{array}
\eeq
Chen et. al. show in \cite{ChenHeYeYuan13-1} that a direct extension of ADMM does not necessarily
converge for solving this type of problem whenever $p \ge 3.$ More precisely, for the case $p\ge 3,$ it is required
that the given coefficient matrices $A_i$ satisfy some orthogonality assumptions. In \cite{HongLuo13-1}, Luo and Hong proposed
a variant of ADMM, namely the proximal ADMM, and proved its asymptotical convergence under the strong convexity assumptions
about the objective function. However, to the best of our knowledge, there does not exist a proof for the convergence for the proximal
ADMM method when the strong convexity assumption is removed. On the other hand, the RPD method, which will be shown to be
equivalent to a randomized version of the proximal ADMM method, exhibits an ${\cal O}(1/N)$ rate of convergence
for solving problem \eqnok{LCP} without requiring
any assumptions about the matrices $A_1,A_2,..,A_p$, as well as the strong convexity assumptions about $f_i$, $i = 1, \ldots, p$.

Let us first formally state these two algorithms. Observing that problem \eqnok{LCP} is equivalent to
\beq \label{lagrange2}
\min_{y \in Y} \max_{x \in X} \left\{ \left\langle {y,b}\right\rangle -\left\langle {y,\tsum_{i=1}^pA_ix_i}\right\rangle-\tsum_{i=1}^pf_i(x_i) \right \},
\eeq
where $Y=\bbr^m$, we can specialize Algorithm~\ref{algRPD} applied to problem \eqnok{lagrange2} as shown Algorithm~\ref{algRPD_LCP3}.
On the other hand, noting that the augmented Lagrangian function of \eqnok{LCP} is given by
\beq
L(x,y)=\left\{ \tsum_{i=1}^pf_i(x_i)+ \left\langle{y, \tsum_{i=1}^pA_ix_i-b} \right\rangle + \tfrac{\rho}{2}\| \tsum_{i=1}^pA_ix_i-b\|^2\right\},
\eeq
we can state the proximal ADMM method for solving problem \eqnok{LCP} as shown in Algorithm~\ref{ADMM}.
It is easy to see that \eqnok{eqn_update_y_lcp3} can be rewritten as
$$y^{t+1}=y^t+\tfrac{1}{\tau_t} \left(\tsum_{i=1}^pA_ix_i^{t+1}-b\right),$$
which implies that
\beq \label{y_new}\bar y^{t+1}=y^{t+1}+\tfrac{q_t}{\tau_t}\left( \tsum_{i=1}^pA_ix_i-b\right).\eeq
In view of \eqnok{y_new}, if only
a randomly selected block $x_{i_t}^{t+1}$ is updated in the Step 2 of the proximal ADMM method instead of all blocks of $x^{t+1}$,
 then the randomized version of \eqnok{eqn_update_x1} and \eqnok{eqn_update_xit} are equivalent in case $\rho=\tfrac{q_{t-1}}{\tau_{t-1}}$.
Therefore, we conclude that Algorithm~\ref{algRPD} applied to problem \eqnok{lagrange2} is equivalent to
a randomized version of the proximal ADMM method for solving linearly constrained problems \eqnok{LCP}.

\begin{algorithm} [H]
    \caption{Randomized primal-dual Methods for problem \eqnok{lagrange2}}
    \label{algRPD_LCP3}
    \begin{algorithmic}
\STATE Let $z^1=(x^1,y^1) \in X\times Y$ and stepsizes $\{\gamma_t\}_{t \ge 1}$, $\{q_t\}_{t \ge 1},$ $\{\tau_t\}_{t \ge 1},$ $\{\eta_t\}_{t \ge 1}.$

\FOR {$t=1, \ldots,N$}

\STATE 1. Generate a random variable $i_t$ from $\{1,\ldots,p \}.$

\STATE 2. Update $y^{t+1}$ and $x^{t+1}$ by

\begin{align}
x_i^{t+1} &=
\begin{cases}
\argmin_{x_{i_t} \in X_{i_t}}{f_{i_t}(x_{i_t})+\langle {\bar y_t, A_{i_t}x_{i_t}}\rangle+\tfrac{\eta_t}{2} \| x_{i_t}-x_{i_t}^t\|_2^2}, & i = i_t,\\
x_{i}^{t}, & i \ne i_t.
\end{cases} \label{eqn_update_xit}\\
y^{t+1}&=\argmin_{y \in Y} {\langle {y,b}\rangle-\langle {y,\tsum_{i=1}^pA_ix_i^{t+1}}\rangle}+\tfrac{\tau_t}{2} \| y-y^t\|_2^2. \label{eqn_update_y_lcp3}\\
\bar y^{t+1}&=q_t (y^{t+1}-y^t)+ y^{t+1}. \label{eqn_x_bar_new}
\end{align}

\ENDFOR



    \end{algorithmic}
\end{algorithm}

 \begin{algorithm} [t]
    \caption{Proximal alternating direction of multiplier methods}
    \label{ADMM}
    \begin{algorithmic}
\STATE Let $z=(x^1,y^1) \in X\times Y$ and stepsizes  $\{\eta_t\}_{t \ge 1}.$

\FOR {$t=1, \ldots,N$}

\STATE Update $y^{t+1}$ and $x^{t+1}$ by
\begin{align}
x_i^{t+1}&=\arg\min\limits_{x_i \in X_i}{f_i(x_i)+\langle {y_t,A_ix_i}\rangle+\rho\langle {\tsum_{j<i}A_ix_i^{t+1}+\tsum_{j\ge i}A_ix_i^{t}-b,A_ix_i}
\rangle+\tfrac{\eta_t}{2} \| x_i-x_i^t\|_2^2}, i=1,\ldots,p. \label{eqn_update_x1}\\
y^{t+1}&=y^t+\rho\left(\tsum_{i=1}^pA_ix_i^{t+1}-b \right). \label{eqn_update_yADMM}
\end{align}
\ENDFOR
    \end{algorithmic}
\end{algorithm}

In order to understand its practical performance for solving the worst-case instances in
\cite{ChenHeYeYuan13-1}, we implement Algorithm~\ref{algRPD} for solving the linearly constrained
problem \eqnok{LCP} with $b=0$ and $f_i(x_i)=0, \;i=1,2,\ldots,p.$ Moreover, we assume that $A_i, i=1,\ldots,p$
are set to $A_1=(1;1;\ldots;1), A_2=(1;\ldots;1;2),...,A_p=(1;2;2;\ldots;2).$ Under the above settings, problem \eqnok{LCP} is equivalent to
a homogenous linear system with $p$ variables
\beq\label{LE}\tsum_{i=1}^p A_ix_i=0,\eeq
where $A_i, i=1,2,\ldots,p$ are nonsingular. Problem \eqnok{LE} has a unique solution $x^*=(0;0;\ldots;0) \in \bbr^n.$
The problem constructed above slightly generalizes the counter example in \cite{ChenHeYeYuan13-1}.
As shown in Table~\ref{table1},
while the original ADMM does not necessarily converge in solving the above problem even with $ p =3$,
Algorithm~\ref{algRPD} converges to the optimal solution $x^*$ for all different values of $p$ that we have tested.
 \begin{table}
\centering
\caption{Results of Algorithm~\ref{algRPD} for solving problem \eqnok{LE}}
\label{table1}
\begin{tabular}{  c  c  c c c }
\hline
p & $\| x^{100}-x^*\|$ &  $\| x^{1,000}-x^*\|$ & $\| x^{10,000}-x^*\|$  &  $x^{100,000}-x^*$ \\ 
\hline
10 & $2.0608$ & $1.1416$ & $0.2674$& $0.0396$  \\

20 &$4.2308$ & $1.1438$ & $1.6588$& $0.4711$  \\

 50 &$7.0277$ & $6.6469$ & $2.2886$& $2.1143$ \\
 \hline
\end{tabular}
\end{table}

\section{Concluding remarks}
In this paper, we present a new randomized algorithm, namely the randomized primal-dual method, for
solving a class of bilinear saddle point problems. Each iteration of the RPD method requires to solve only one subproblem rather than all subproblems as in
the original primal-dual algorithms. The RPD method does not require strong convexity assumptions about the objective function
and/or boundedness assumptions about the feasible sets. Moreover, based on a new primal-dual termination criterion,
we show that this algorithm exhibits an ${\cal O}(1/N)$ rate of convergence for both bounded and unbounded saddle point problems and
and ${\cal O}(1/N^2)$ rate of convergence for smooth saddle point problems. Extension for the non-Euclidean
setting and the relation to the ADMM method have also been discussed in this paper.

It is worth noting that there exist a few possible extensions of this work. Firstly, for the case when $h(x)$ is not necessarily simple, but a general smooth convex,
one can modify \eqnok{eqn_update_x} in Step 2 of Algorithm~\ref{algRPD} by replacing $h(x)$ with its linear approximation as suggested in \cite{CheLanOu13-1}. Secondly, this paper focuses on the case when the dual space has multiple blocks. However, it is possible to apply block decomposition for both the primal and dual spaces whenever the feasible sets $X$ and $Y$ are decomposable. Finally, it will be interesting to see if the rate of convergence
for the RPD methods can be further improved by using non-uniform distribution for the random variables $i_t$.

\bibliographystyle{plain}
\bibliography{glan-bib1}

\newcommand{\noopsort}[1]{} \newcommand{\printfirst}[2]{#1}
  \newcommand{\singleletter}[1]{#1} \newcommand{\switchargs}[2]{#2#1}
\begin{thebibliography}{10}

\bibitem{BeckTet13-1}
A.~Beck and L.~Tetruashvili.
\newblock On the convergence of block coordinate descent type methods.
\newblock Technical report.
\newblock submitted to {\sl SIAM Journal on Optimization}.

\bibitem{BlHeGa07-1}
D.~Blatt, A.~Hero, and H.~Gauchman.
\newblock A convergent incremental gradient method with a constant step size.
\newblock {\em SIAM Journal on Optimization}, 18(1):29--51, 2007.

\bibitem{boyd2011distributed}
Stephen Boyd, Neal Parikh, Eric Chu, Borja Peleato, and Jonathan Eckstein.
\newblock Distributed optimization and statistical learning via the alternating
  direction method of multipliers.
\newblock {\em Foundations and Trends{\textregistered} in Machine Learning},
  3(1):1--122, 2011.

\bibitem{Breg67}
L.M. Bregman.
\newblock The relaxation method of finding the common point convex sets and its
  application to the solution of problems in convex programming.
\newblock {\em USSR Comput. Math. Phys.}, 7:200--217, 1967.

\bibitem{burachik1997enlargement}
Regina~S Burachik, Alfredo~N Iusem, and Benar~Fux Svaiter.
\newblock Enlargement of monotone operators with applications to variational
  inequalities.
\newblock {\em Set-Valued Analysis}, 5(2):159--180, 1997.

\bibitem{ChamPoc11-1}
A.~Chambolle and T.~Pock.
\newblock A first-order primal-dual algorithm for convex problems with
  applications to imaging.
\newblock {\em J. Math. Imaging Vision}, 40:120--145, 2011.

\bibitem{ChenHeYeYuan13-1}
Caihua Chen, Bingsheng He, Yinyu Ye, and Xiaoming Yuan.
\newblock The direct extension of admm for multi-block convex minimization
  problems is not necessarily convergent.
\newblock {\em Optimization Online}, 2013.

\bibitem{CheLanOu14-1}
Y.~Chen, G.~Lan, and Y.~Ouyang.
\newblock Accelerated schemes for a class of variational inequalities.
\newblock {\em Mathematical Programm, Series B}, 2014.
\newblock submitted.

\bibitem{CheLanOu13-1}
Y.~Chen, G.~Lan, and Y.~Ouyang.
\newblock Optimal primal-dual methods for a class of saddle point problems.
\newblock {\em SIAM Journal on Optimization}, 24(4):1779--1814, 2014.

\bibitem{DangLan13-1}
C.~D. Dang and G.~Lan.
\newblock Stochastic block mirror descent methods for nonsmooth and stochastic
  optimization.
\newblock {\em SIAM Journal on Optimization}, 2015.
\newblock to appear.

\bibitem{DouglasRachford56-1}
Jr. Douglas, Jim and Jr. Rachford, H.~H.
\newblock On the numerical solution of heat conduction problems in two and
  three space variables.
\newblock {\em Transactions of the American Mathematical Society}, 82(2):pp.
  421--439, 1956.

\bibitem{EcBe92-1}
Jonathan Eckstein and DimitriP. Bertsekas.
\newblock On the douglas-rachford splitting method and the proximal point
  algorithm for maximal monotone operators.
\newblock {\em Mathematical Programming}, 55(1-3):293--318, 1992.

\bibitem{esser2010general}
E.~Esser, X.~Zhang, and T.F. Chan.
\newblock A general framework for a class of first order primal-dual algorithms
  for convex optimization in imaging science.
\newblock {\em SIAM Journal on Imaging Sciences}, 3(4):1015--1046, 2010.

\bibitem{gabay1976dual}
D.~Gabay and B.~Mercier.
\newblock A dual algorithm for the solution of nonlinear variational problems
  via finite element approximation.
\newblock {\em Computers \& Mathematics with Applications}, 2(1):17--40, 1976.

\bibitem{Gabay83-1}
Daniel Gabay.
\newblock Chapter ix applications of the method of multipliers to variational
  inequalities.
\newblock {\em Studies in mathematics and its applications}, 15:299--331, 1983.

\bibitem{GhaLan12-2a}
S.~Ghadimi and G.~Lan.
\newblock Optimal stochastic approximation algorithms for strongly convex
  stochastic composite optimization, {I}: a generic algorithmic framework.
\newblock {\em SIAM Journal on Optimization}, 22:1469--1492, 2012.

\bibitem{glowinski1975approximation}
Roland Glowinski and A~Marroco.
\newblock Sur l'approximation, par {\'e}l{\'e}ments finis d'ordre un, et la
  r{\'e}solution, par p{\'e}nalisation-dualit{\'e} d'une classe de
  probl{\`e}mes de dirichlet non lin{\'e}aires.
\newblock {\em ESAIM: Mathematical Modelling and Numerical
  Analysis-Mod{\'e}lisation Math{\'e}matique et Analyse Num{\'e}rique},
  9(R2):41--76, 1975.

\bibitem{GoMa12-1}
D.~Goldfarb and S.~Ma.
\newblock Fast multiple-splitting algorithms for convex optimization.
\newblock {\em SIAM Journal on Optimization}, 22(2):533--556, 2012.

\bibitem{GoMaSc13-1}
Donald Goldfarb, Shiqian Ma, and Katya Scheinberg.
\newblock Fast alternating linearization methods for minimizing the sum of two
  convex functions.
\newblock {\em Mathematical Programming}, 141(1-2):349--382, 2013.

\bibitem{HeYu12-1}
B.~He and X.~Yuan.
\newblock On the \$o(1/n)\$ convergence rate of the douglasñrachford
  alternating direction method.
\newblock {\em SIAM Journal on Numerical Analysis}, 50(2):700--709, 2012.

\bibitem{he2012on}
Bingsheng He and Xiaoming Yuan.
\newblock On the o(1/n) convergence rate of the douglas-rachford alternating
  direction method.
\newblock {\em SIAM Journal on Numerical Analysis}, 50(2):700--709, 2012.

\bibitem{HeJudNem14-1}
N.~He, A.~Juditsky, and A.~Nemirovski.
\newblock Mirror prox algorithm for multi-term composite minimization and
  semi-separable problems.
\newblock {\em Computational Optimization and Applications}, 2014.
\newblock submitted.

\bibitem{HongLuo13-1}
M.~{Hong} and Z.-Q. {Luo}.
\newblock {On the Linear Convergence of the Alternating Direction Method of
  Multipliers}.
\newblock {\em ArXiv e-prints}, August 2012.

\bibitem{Lan10-3}
G.~Lan.
\newblock An optimal method for stochastic composite optimization.
\newblock {\em Mathematical Programming}, 133(1):365--397, 2012.

\bibitem{LaLuMo11-1}
G.~Lan, Z.~Lu, and R.~D.~C. Monteiro.
\newblock Primal-dual first-order methods with {${\cal O}(1/\epsilon)$}
  iteration-complexity for cone programming.
\newblock {\em Mathematical Programming}, 126:1--29, 2011.

\bibitem{LanMon09-1}
G.~Lan and R.~D.~C. Monteiro.
\newblock Iteration-complexity of first-order augmented lagrangian methods for
  convex programming.
\newblock {\em Mathematical Programming}.

\bibitem{LanMon13-1}
G.~Lan and R.~D.~C. Monteiro.
\newblock Iteration-complexity of first-order penalty methods for convex
  programming.
\newblock {\em Mathematical Programming}, 138:115--139, 2013.

\bibitem{LevLew10-1}
D.~Leventhal and A.~S. Lewis.
\newblock Randomized methods for linear constraints: Convergence rates and
  conditioning.
\newblock {\em Mathematics of Operations Research}, 35:641--654, 2010.

\bibitem{LiMe79-1}
P.~L. Lions and B.~Mercier.
\newblock Splitting algorithms for the sum of two nonlinear operators.
\newblock {\em SIAM Journal on Numerical Analysis}, 16(6):pp. 964--979, 1979.

\bibitem{LuXiao13-1}
Z.~Lu and L.~Xiao.
\newblock On the complexity analysis of randomized block-coordinate descent
  methods.
\newblock Manuscript, 2013.

\bibitem{MahdaviJ13}
Mehrdad Mahdavi and Rong Jin.
\newblock Mixedgrad: An {O(1/T)} convergence rate algorithm for stochastic
  smooth optimization.
\newblock {\em CoRR}, abs/1307.7192, 2013.

\bibitem{MonSva09-1}
R.D.C. Monteiro and B.F. Svaiter.
\newblock On the complexity of the hybrid proximal extragradient method for the
  iterates and the ergodic mean.
\newblock Manuscript, School of ISyE, Georgia Tech, Atlanta, GA, 30332, USA,
  March 2009.

\bibitem{MonSva10-3}
R.D.C. Monteiro and B.F. Svaiter.
\newblock Complexity of variants of tsengÔøΩs modified f-b splitting and
  korpelevich's methods for hemi-variational inequalities with applications to
  saddle-point and convex optimization problems.
\newblock Manuscript, School of ISyE, Georgia Tech, Atlanta, GA, 30332, USA,
  June 2010.

\bibitem{MonSva10-1}
R.D.C. Monteiro and B.F. Svaiter.
\newblock On the complexity of the hybrid proximal projection method for the
  iterates and the ergodic mean.
\newblock {\em SIAM Journal on Optimization}, 20:2755--2787, 2010.

\bibitem{monteiro2013iteration}
Renato~DC Monteiro and Benar~F Svaiter.
\newblock Iteration-complexity of block-decomposition algorithms and the
  alternating direction method of multipliers.
\newblock {\em SIAM Journal on Optimization}, 23(1):475--507, 2013.

\bibitem{Nem05-1}
A.~S. Nemirovski.
\newblock Prox-method with rate of convergence $o(1/t)$ for variational
  inequalities with lipschitz continuous monotone operators and smooth
  convex-concave saddle point problems.
\newblock {\em SIAM Journal on Optimization}, 15:229--251, 2005.

\bibitem{NJLS09-1}
A.~S. Nemirovski, A.~Juditsky, G.~Lan, and A.~Shapiro.
\newblock Robust stochastic approximation approach to stochastic programming.
\newblock {\em SIAM Journal on Optimization}, 19:1574--1609, 2009.

\bibitem{Nest05-1}
Y.~E. Nesterov.
\newblock Smooth minimization of nonsmooth functions.
\newblock {\em Mathematical Programming}, 103:127--152, 2005.

\bibitem{Nest10-1}
Y.~E. Nesterov.
\newblock Efficiency of coordinate descent methods on huge-scale optimization
  problems.
\newblock Technical report, Center for Operations Research and Econometrics
  (CORE), Catholic University of Louvain, Feburary 2010.

\bibitem{Nest12-1}
Y.~E. Nesterov.
\newblock Subgradient methods for huge-scale optimization problems.
\newblock Technical report, Center for Operations Research and Econometrics
  (CORE), Catholic University of Louvain, Feburary 2012.

\bibitem{OuCheLanPas14-1}
Y.~Ouyang, Y.~Chen, G.~Lan, and E.~Pasiliao.
\newblock An accelerated linearized alternating direction method of
  multipliers.
\newblock {\em SIAM Journal on Imaging Sciences}, 2014.
\newblock to appear.

\bibitem{Rich12-1}
P.~Richt\'{a}rik and M.~Tak\'{a}\u{c}.
\newblock Iteration complexity of randomized block-coordinate descent methods
  for minimizing a composite function.
\newblock {\em Mathematical Programming}, 2012.
\newblock to appear.

\bibitem{Rocka76-1}
R.~Rockafellar.
\newblock Monotone operators and the proximal point algorithm.
\newblock {\em SIAM Journal on Control and Optimization}, 14(5):877--898, 1976.

\bibitem{SchRouBac13-1}
M.~Schmidt, N.~L. Roux, and F.~Bach.
\newblock Minimizing finite sums with the stochastic average gradient.
\newblock Technical report, September 2013.

\bibitem{ShaZhang15-1}
S.~Shalev-Shwartz and T.~Zhang.
\newblock Accelerated proximal stochastic dual coordinate ascent for
  regularized loss minimization.
\newblock {\em Mathematical Programming}, 2015.
\newblock to appear.

\bibitem{ShalevZhang13-1}
Shai Shalev-Shwartz and Tong Zhang.
\newblock Stochastic dual coordinate ascent methods for regularized loss.
\newblock {\em J. Mach. Learn. Res.}, 14(1):567--599, February 2013.

\bibitem{Suzuki13}
T.~{Suzuki}.
\newblock {Stochastic Dual Coordinate Ascent with Alternating Direction
  Multiplier Method}.
\newblock {\em ArXiv e-prints}, November 2013.

\bibitem{LinMaZhang14-1}
Shuzhong~Zhang Tianyi~Lin, Shiqian~Ma.
\newblock {On the Global Linear Convergence of the ADMM with Multi-Block
  Variables}.
\newblock {\em ArXiv e-prints}, 2014.

\bibitem{Yuchen14}
Yuchen Zhang and Lin Xiao.
\newblock Stochastic primal-dual coordinate method for regularized empirical
  risk minimization.
\newblock Manuscript, September 2014.

\end{thebibliography}
\end{document}